\documentclass{amsart}

\usepackage[utf8]{inputenc}
\usepackage[numbers]{natbib} 
\usepackage{hyperref}
\usepackage{float}
\usepackage{amsmath, amssymb, amsthm}
\usepackage{geometry}
\usepackage{graphicx}
\usepackage{mathrsfs} 
\usepackage{tikz-cd}

\newtheorem{theorem}{Theorem}[section]

\newtheorem{lemma}[theorem]{Lemma}
\newtheorem{corollary}[theorem]{Corollary}
\theoremstyle{definition}
\newtheorem{definition}[theorem]{Definition}
\theoremstyle{remark}

\theoremstyle{plain}
\newtheorem{proposition}[theorem]{Proposition}

\theoremstyle{definition}
\newtheorem{example}{Example}[section]
\title{Equivariant Maximal Cohen-Macaulay sheaves on the minimal orbit closures}
\author{Shang Xu}

\begin{document}

\begin{abstract}
In this paper, we study maximal Cohen-Macaulay sheaves on closures of minimal nilpotent orbits in simple Lie algebras. For singularities of type $A_n$, we first classify vector bundles on their symplectic resolutions whose pushforwards are maximal Cohen-Macaulay. We then construct equivariant maximal Cohen-Macaulay sheaves via irreducible representations of the stabilizer group. We compare these two approaches in the case of maximal Cohen-Macaulay Weil divisors, and extend the equivariant construction to the classical types $B_n$, $C_n$, and $D_n$. Finally, we formulate the construction for an arbitrary simple Lie algebra and carry it out explicitly in the exceptional cases.
\end{abstract}
\maketitle
\tableofcontents
\section{Introduction}

Let $G$ be a simply connected complex Lie group with simple Lie algebra $\mathfrak{g}$. Under the adjoint action of $G$ on $\mathfrak{g}$, the set of nilpotent elements $\mathcal{N}$ decomposes into finitely many $G$-orbits. Among the nonzero $G$-orbits in $\mathcal {N}$, there is a unique orbit of minimal dimension. This is called the minimal nilpotent orbit, usually denoted by $\mathcal{O}_{\min}$. Its closure is $\overline{\mathcal O_{\min}}=\mathcal O_{\min}\cup\{0\}$, where $0$ is an isolated singularity. With Kostant–Kirillov symplectic form on $\mathcal O_{\min}$, this singularity is a \textit{symplectic singularity} in Beauville's sense in \cite{Beauville2000SymplecticSingularities}. It's also mentioned in this paper that those minimal orbit closures are indeed analytic prototypes for isolated symplectic singularities with smooth projective tangent cones.

Maximal Cohen-Macaulay sheaves play a central role in the study of singular algebraic varieties. They arise naturally as higher syzygies in free resolutions and may be regarded as the sheaves that are, in a precise sense, closest to vector bundles on smooth spaces. Maximal Cohen-Macaulay sheaves on surface symplectic singularities, also known as ADE singularities, were classified in \cite{ArtinVerdierEsnault1985}. In \cite{xu2026constructingmaximalcohenmacaulaysheaves}, we constructed many families of maximal Cohen-Macaulay sheaves on $\mathcal{N}_{n+1}$, the minimal nilpotent orbit closure of type $A_n$, showing the complexity of this classification problem in higher symplectic singularities. 

In the first part of this paper, we further develop that method and classify vector bundles on symplectic resolutions of isolated symplectic singularities whose pushforwards are maximal Cohen-Macaulay sheaves. The main result of this part, stated in Section~\ref{MCMonsymplecticresolution}, is the following: 
\begin{theorem}[Theorem \ref{isolatedsymplecticMCM}]
    Suppose $X$ is a $2n$-dimensional variety with only isolated singularities and $\pi:\widetilde{X}\rightarrow X$ is a symplectic resolution. For a vector bundle $\mathcal{F}$ on $\widetilde{X}$, $\pi_*\mathcal{F}$ is maximal Cohen-Macaulay if and only if
\begin{enumerate}
\item $R^i\pi_*\mathcal{F}=R^i\pi_*\mathcal{F}^\vee=0$ for $1\leq i\leq n-1$,
\item $e_{\mathcal{F}}:(R^{n}\pi_*\mathcal{F})'\rightarrow R^n\pi_*\mathcal{F}^\vee$ is an isomorphism.
\end{enumerate}
\end{theorem}

We use this theorem to check divisors on $\mathcal{N}_{n+1}$ and conclude that in $\operatorname{Cl}\mathcal{N}_{n+1}=\mathbb{Z}$, with a generator $D$, $\mathcal{O}(kD)$ is maximal Cohen-Macaulay if and only if $-n\leq k\leq n$. These maximal Cohen-Macaulay divisors are exactly what we have constructed in \cite{xu2026constructingmaximalcohenmacaulaysheaves}. We cite the theorem here,

Although the classification of all maximal Cohen-Macaulay modules on $\mathcal{N}_{n+1}$ for $n\geq 3$ appears difficult, the equivariant setting is more tractable. Restricting to $SL(n+1)$-equivariant sheaves, Theorem \ref{localseqMCM} and the Borel-Weil-Bott theorem translates the maximal Cohen-Macaulay condition into an explicit combinatorial criterion on weights. We give a criterion determining when the sheaf induced by an irreducible representation of the stabilizer subgroup is maximal Cohen-Macaulay. 
\begin{theorem}[Theorem \ref{SLMCM}]
    Suppose $\mathcal{O}_{\min}$ is the minimal orbit in $\mathfrak{sl}_{n+1}$ $(n\geq 2)$ with respect to the conjugation action of $SL(n+1)$ and $H$ is the stabilizer. Given an equivariant bundle $\mathcal{E}_V=SL(n+1)\times_H V^\vee$ corresponding to an irreducible representation $V$ of $H$ determined by a weight $(\lambda,\lambda_1,\lambda_2...\lambda_{n-1})$ in $\mathbb{Z}^n/\mathbb{Z}(0,1,...,1)$ with $\lambda_1\geq\lambda_2\geq...\geq\lambda_{n-1}$, the pushforward $i_*\mathcal{E}_V$ is maximal Cohen-Macaulay on orbit closure $\overline{\mathcal{O}_{\min}}=\mathcal{N}_{n+1,1}$ if and only if 
\begin{enumerate} 
\item If $\lambda_1=\lambda_2=\cdots=\lambda_{n-1}=\mu$, then $2\mu-n\leq\lambda\leq2\mu+n$.
\item If $\lambda_1>\lambda_{n-1}$, there is $\lambda_{n-1}+\lambda_M-M\leq\lambda\leq\lambda_1+\lambda_m+n-m$ and $$\lambda\in\displaystyle\bigcap_{\alpha\in A}\{\lambda_i+\alpha-i|1\leq i \leq n-1\}
$$
\end{enumerate}
where $M=\max\{i|1\leq i\leq n-1, \lambda_i=\lambda_1\}$, $m=\min\{i|1\leq i\leq n-1, \lambda_i=\lambda_{n-1}\}$ and $$A=\{N|N\in\mathbb{Z}, \lambda_{n-1}+1<N<\lambda_{1}+n-1, N\neq\lambda_{i}+n-i\ \text{for any}\ 1\leq i\leq n-1, N\neq\frac{\lambda+n}{2}\} $$ 
\end{theorem}

This method can be generalized to the classification problem of matrix group equivariant maximal Cohen-Macaulay sheaves on $B_n$, $C_n$, $D_n$ type minimal orbit closures. We also cite the main results here:

\begin{theorem}[$C_n$ type, Theorem \ref{Ctypetheorem}]
Suppose $\mathcal{O}_{\min}$ is the minimal orbit in $\mathfrak{sp}_{2n}$ $(n\geq 2)$ with respect to the conjugation action of $Sp(2n)$ and $H$ is the stabilizer. An equivariant bundle $\mathcal{E}_V=Sp(2n)\times_H V^\vee$ corresponding to an irreducible representation $V$ of $H$ is determined by the weight $(\lambda_0,\lambda_1,\lambda_2...\lambda_{n-1})$ in $H$'s lattice $\mathbb{Z}/2\times\mathbb{Z}^{n-1}$. The pushforward $i_*\mathcal{E}_V$ is maximal Cohen-Macaulay on the orbit closure $\overline{\mathcal{O}_{\min}}$ if and only if 
for any element $k$ in the following set
$$A=\{N|N\in\mathbb{Z}, \ |N|<\lambda_{1}+n-1,\  N\neq 0,\   \ N\neq \pm(\lambda_{i}+n-i)\ \  \text{for any}\ \   1\leq i\leq n-1\}
$$
we have $k\not\equiv n-\lambda_0 (\operatorname{mod} 2)$.
\end{theorem}

\begin{theorem}[$D_n$ type, Theorem \ref{Dtypetheorem}]
    Suppose $\mathcal{O}_{\min}$ is the minimal orbit in $\mathfrak{so}_{2n}$ $(n\geq 3)$ with respect to the conjugation action of $SO(2n)$ and $H$ is the stabilizer. An equivariant bundle $\mathcal{E}_V=SO(2n)\times_H V^\vee$ corresponding to an irreducible representation $V$ of $H$ is determined by a dominant weight $(\lambda,\lambda_1,\lambda_2...\lambda_{n-2})$ with $\lambda\geq 0$ and $\lambda_1\geq...\geq\lambda_{n-3}\geq|\lambda_{n-2}|$ in $H$'s lattice $\mathbb{Z}^{n}$. The pushforward $i_*\mathcal{E}_V$ is maximal Cohen-Macaulay on the orbit closure $\overline{\mathcal{O}_{\min}}$ if and only if
$$\lambda\in\bigcap_{k\in A}\left(\bigcup_{i=1}^{n-2}\{k-\lambda_i+i,k+2n-4+\lambda_i-i\}\cup\{2k+2n-3\}\right)$$
where 
$$A=\{N|N\in\mathbb{Z},\  -\lambda_1-2n+4\leq N\leq \lambda+\lambda_1-1, \ N\neq\lambda_i-i-1,\ N\neq-\lambda_i-2n+3+i\ \text{for all}\ i\}
$$
\end{theorem}

\begin{theorem}[$B_n$ type, Theorem \ref{Btypetheorem}]
    Suppose $\mathcal{O}_{\min}$ is the minimal orbit in $\mathfrak{so}_{2n+1}$ $(n\geq 3)$ with respect to the conjugation action of $SO(2n+1)$ and $H$ is the stabilizer. An equivariant bundle $\mathcal{E}_V=SO(2n+1)\times_H V^\vee$ corresponding to an irreducible representation $V$ of $H$ is determined by a dominant weight $(\lambda,\lambda_1,\lambda_2...\lambda_{n-2})$ with $\lambda\geq 0$ and $\lambda_1\geq...\geq\lambda_{n-3}\geq\lambda_{n-2}\geq0$ in $H$'s lattice $\mathbb{Z}^{n}$. The pushforward $i_*\mathcal{E}_V$ is maximal Cohen-Macaulay on the orbit closure $\overline{\mathcal{O}_{\min}}$ if and only if 
for any element $k$ in the following set
$$\lambda\in\bigcap_{k\in A}\left(\bigcup_{i=1}^{n-2}\{k-\lambda_i+i,k+\lambda_i+2n-3-i\}\cup\{2k+2n-2\}\right)$$
where 
$$A=\{N|N\in\mathbb{Z}, -\lambda_1-2n+3\leq N\leq\lambda+\lambda_1-1, N\neq \lambda_i-1-i,\ N\neq-\lambda_i-2n+2+i\ \text{for all}\ i\}
$$
\end{theorem}

Finally, in order to generalize this story to exceptional cases, we extend the study of equivariant maximal Cohen-Macaulay sheaves on minimal nilpotent orbit closures from the classical cases to a uniform, coordinate-free framework in \ref{abstractanalysis}. Starting from a highest-root vector $e_\theta$ in a simple Lie algebra $\mathfrak{g}$, we describe the minimal orbit as $\mathcal{O}_{\min}=G/H$, where $H$ is the stabilizer of $e_\theta$, and relate it to the parabolic quotient $G/P$ obtained by stabilizing the line $\mathbb{C}e_\theta$. We show the reason why it's always possible to find a $1$-dimensional torus $S$ such that $SH=P$, $S\cap H\cong\mu_2$, and that $S$ commutes with the Levi factor of $H$. We also show how to identify the simply connected Levi factor of $H$ and its root system $\Phi^\natural$ through the affine Dynkin diagram. Using these abstract arguments, we generalize our $SO$-equivariant results to $\operatorname{Spin}$-equivariant cases and state them there. Readers who are primarily interested in the abstract framework may proceed directly to Section \ref{abstractanalysis}. Then, in Section \ref{exceptional} we explain the favorable structural properties of the exceptional cases. Building on this analysis, Section \ref{exceptionalclassifications} applies the method to the exceptional types $E_6,E_7,E_8,F_4$, and $G_2$, producing explicit numerical lists of dominant weights whose associated equivariant sheaves push forward to maximal Cohen-Macaulay sheaves on $\overline{\mathcal{O}_{\min}}$.

\textbf{Conventions.} Throughout this paper, we work over the field $\mathbb{C}$. All varieties are assumed to be irreducible and quasi-projective unless otherwise stated. 

\textbf{Acknowledgments.} I would like to thank Professor David Eisenbud for his valuable comments. Thanks Professor Vera Serganova for inspiring me to consider the equivariant case. Thanks also Wenqing Wei, Jerry Yang and Peisheng Yu for useful discussions.

\section{Singularities and Maximal Cohen-Macaulay sheaves}

In this chapter, we work on constructing maximal Cohen-Macaulay sheaves on some singular varieties. Our main interest is in those varieties with isolated singularities. If the singularity is also \textit{symplectic}, while some methods are developed in \cite{xu2026constructingmaximalcohenmacaulaysheaves} for constructing purpose, in section 3 we will state a criterion to determine if a vector bundle on the resolution of an isolated symplectic singularity pushes forward to a maximal Cohen-Macaulay sheaf. 

As all symplectic isolated singularities with symplectic resolutions are very close to $A_n$ singularities (we explain this point in 2.2), special treatment should be applied to these nilpotent orbit singularities. If one again considers the equivariant sheaves with respect to conjugation action of $SL(n)$, a full classification is possible, at least for those sheaves induced by an irreducible representation of the stabilizer, which is the main result in the following chapters. This also generalizes to those minimal orbits in other types of semisimple Lie algebras, and therefore it is worth briefly discussing the structures of those isolated singularities in the minimal orbit closures here too. One should notice that although all minimal orbits have symplectic singularities, only $A_n$'s minimal orbits admit symplectic resolutions (cf. \cite{Fu2003Symplectic}).
\subsection{Symplectic Singularities}
We start with the definition of symplectic singularity.
\begin{definition}[Symplectic singularity]
A variety $X$ has a \emph{symplectic singularity} at a point $p \in X$ if there exists an open neighborhood $U \subset X$ of $p$ such that:
\begin{enumerate}
\item[(a)] $U$ is normal; 
\item[(b)] The smooth locus $U_{\mathrm{reg}}$ admits an algebraic symplectic $2$-form $\varphi$;
\item[(c)] For any resolution of singularities $f : \widetilde{U} \to U$, the pullback $f^{*}\varphi$ defined on $f^{-1}(U_{\mathrm{reg}})$ extends to an algebraic $2$-form on $\widetilde{U}$.
\end{enumerate}
\end{definition}
In this section, we assume $X=\operatorname{Spec}\mathcal{O}$ is an affine singular variety with only symplectic singularities and $\pi: \widetilde{X}\rightarrow X$ is a resolution of $X$. We assume this resolution is \textit{symplectic}, i.e., there is an algebraic symplectic form on $X^{reg}$ which extends to a global symplectic form on $\widetilde{X}$. One can show that in the case $X$ is affine this is the same as saying $\pi: \widetilde{X}\rightarrow X$ is a \textit{crepant} resolution, i.e. the pullback of canonical class of $X$ is the canonical class of $\widetilde{X}$.

An important property of symplectic singularities is 

\begin{proposition}[\cite{Beauville2000SymplecticSingularities} 1.3]\label{Gorat}
A symplectic singularity is rational Gorenstein.
\end{proposition}

Therefore, by the Gorenstein property, $\mathcal{O}_{X}$ is the dualizing sheaf on $X$, and by the definition of rational singularities, we have $\mathbf{R}\pi_{*}\mathcal{O}_{\widetilde{{X}}}=\mathcal{O}_{X}$, or equivalently, $R^{>0}\pi_{*}\mathcal{O}_{\widetilde{X}}=0$. The natural map $\mathcal{O}_{X}\rightarrow\mathcal{O}_{\widetilde{X}}$ is an isomorphism. Since $X$ is assumed to be affine, this is also equivalent to saying $H^{>0}(\mathcal{O}_{\widetilde{X}})=0$ and $H^{0}(\mathcal{O}_{\widetilde{X}})=H^{0}(\mathcal{O}_{{X}})$.

There are many examples of symplectic resolutions:

\begin{example}[ADE surface singularities]
There is a well known family of symplectic singularities: the \textit{rational double points} or \textit{ADE singularities}, and they are constructed as follows. Suppose $X=\mathbb{C}^2/G$ where $G$ is a finite group in $SL(2,\mathbb{C})=Sp(2,\mathbb{C)}$. Then $X$ inherits a symplectic form from $\mathbb{C}^2$, and there is one isolated singular point on $X$ that is symplectic. In \cite{ItoNakamura1996}, a canonical resolution is given as the $G$-Hilbert scheme: 
$$
G\text{-Hilb}(\mathbb{C}^2)\rightarrow\mathbb{C}^2/G
$$

\end{example}
\begin{example}
The generalization of ADE resolutions are the symplectic resolutions given by Nakajima quiver varieties constructed in \cite{Nakajima1994Instantons}. There Nakajima defined a resolution $\mathcal{M}^{\theta}_\lambda(v,w)\rightarrow  \mathcal{M}^{0}_\lambda(v,w)$ for each quiver $Q$ and dimension vector $v,w$ by the natural map between the symplectic reductions of ordinary quiver varieties (from GIT quotient to categorical quotient). When taking $Q$ to be the McKay graph of a finite group $G$ in $SL(2,\mathbb{C})$ and take $v$ to be the dimension vector of different irreducible representations of $G$, this resolution gives the ADE surface resolution above. One can refer \cite{Ginzburg2010Nakajima} for more detailed constructions or \cite{xu2026constructingmaximalcohenmacaulaysheaves} for a much shorter introduction.
\end{example}

\begin{example}[Springer resolution]
Let $G$ be a connected complex semisimple algebraic group with Lie algebra $\mathfrak{g}$, and let $\mathcal{N}\subset\mathfrak{g}$ be the nilpotent cone. Let $\mathcal{B}$ be the flag variety of $G$. The cotangent bundle $T^*\mathcal{B}$ carries its canonical symplectic form and admits a proper $G$-equivariant morphism
\[
\pi:T^*\mathcal{B}\to\mathcal{N},
\qquad
(x,\mathfrak{b})\mapsto x,
\]
where we identify
\[
T^*\mathcal{B}=\{(x,\mathfrak{b})\in\mathfrak{g}\times\mathcal{B}\mid x\in\mathfrak{b},\; x\text{ nilpotent}\}.
\]
The map $\pi$ is proper, birational, and an isomorphism over the regular nilpotent orbit. Hence, $\pi$ is a resolution of singularities of $\mathcal{N}$. Moreover, the symplectic form on $T^*\mathcal{B}$ restricts to the Kostant--Kirillov form on smooth nilpotent orbits, so $\mathcal{N}$ has symplectic singularities in the sense of Beauville, and $\pi$ is a symplectic resolution.
\end{example}

The $A_n$ type Springer resolution can be constructed as Nakajima quiver varieties of $A_n$ Dynkin quiver with one single framing vector space at one end, the details can be found in the well-known paper \cite{Nakajima1994Instantons}.

\subsection{Minimal nilpotent orbit closures}\label{miniorbit}
Minimal orbits closures $\overline{\mathcal{O}_{\min}}$ in simple Lie algebras are important examples and standard analytic model among all isolated symplectic singularities with smooth projective cones, as the following theorem from \cite{Beauville2000SymplecticSingularities} shows,
\begin{proposition}\label{standardisosymp}
Let $(X,p)$ be an isolated symplectic singularity whose projective tangent cone is smooth.
Then $(X,p)$ is analytically isomorphic to $(\overline{\mathcal{O}}_{\min},0)$
for some simple complex Lie algebra.
\end{proposition}

The constructions of Nakajima quiver varieties of $A_n$ type Dynkin quiver gives symplectic resolutions for all nilpotent orbit closures in $\mathfrak{sl}_{n+1}$. In particular, the minimal orbits $\mathcal{N}_{n+1,1}$, which consists of all nilpotent $(n+1)\times(n+1)$ matrices of rank $\leq 1$ has a symplectic resolution. The resolution can be easily written as follows
$$T^*\mathbb{P}^n\rightarrow\mathcal{N}_{n+1,1},\quad (i,j)\mapsto i j 
$$
where we identify $T^*\mathbb{P}^n$ by all pairs $\{(i,j)|i:\mathbb{C}\rightarrow\mathbb{C}^{n+1}$, $j:\mathbb{C}^{n+1}\rightarrow\mathbb{C}$, $ji=0\}$. Now those nilpotent rank 1 matrices have a single Jordan canonical form $E_{12}$, so all such matrices are conjugate to each other. Therefore, there is a single $GL(n)$-orbit. Using the language of partitions, we see that this orbit corresponds to $[2,1^{n-2}]$. Then we see that $\mathcal{N}_{n+1,1}-\{0\}$ is smooth and $0$ is an isolated singularity admitting a symplectic resolution.

It is natural to ask whether the minimal orbit closures of other types of simple Lie algebra also have symplectic resolutions. The answer is negative. The conditions for nilpotent orbit closures of types B,C,D admitting symplectic resolution are summarized in \cite{Fu2003Symplectic} as follows:

\begin{proposition}\label{orbitsymp}[\cite{Fu2003Symplectic} Proposition 3.21, 3.20]
Suppose $\mathfrak{g}$ is a complex simple Lie algebra of type $B,C,D$ and $\mathcal{O}$ be a nilpotent orbit in $\mathfrak{g}$ corresponding to the partition $\mathbf{d} = [d_1, \dots, d_N]$. Then $\overline{\mathcal{O}}$ admits a symplectic resolution if and only if
\begin{enumerate}
\item $\mathfrak{g}=\mathfrak{sp}_{2n}$ and there exists an even number $q \ge 0$ such that the first $q$ parts $d_1, \dots, d_q$ are odd and the other parts are even.
\item $\mathfrak{g}=\mathfrak{so}_{2n+1}$ and there exists an odd number $q \ge 0$ such that the first $q$ parts $d_1, \dots, d_q$ are odd and the other parts are even.
\item $\mathfrak{g}=\mathfrak{so}_{2n}$ and either there exists some even number $q \ne 2$ such that the first $q$ parts of $\mathbf{d}$ are odd and the others are even, or there exist exactly two odd parts which are at the positions $2k-1$ and $2k$ in the partition $\mathbf{d}$ for some $k$;
\end{enumerate}
\end{proposition}
One checks easily that the minimal orbits of type $B_n, C_n, D_n$ correspond to the partitions $(2,2,1^{2n-3})$, $(2,1^{2n-2})$, $(2,2,1^{2n-4})$ respectively, and none of these partitions satisfy the conditions above, so they do not admit symplectic resolution.
    
Indeed, isolated symplectic singularities are all analytically look like ADE surface singularities or $A_n$ singularities, as the following result shows.
\begin{proposition}\label{sympdim}[\cite{ChoMiyaokaShepherdBarron2002}, Theorem 8.3]
Let $\hat{Z}$ be a normal projective variety of dimension $2n$
with a single isolated singularity.
Assume that there exists a symplectic resolution
$\pi : Z \to \hat{Z}$; in other words, $\pi$ is a birational morphism
from the smooth projective complex symplectic variety $Z$ onto $\hat{Z}$.
Then:
\begin{enumerate}
\item The exceptional locus $E \subset \hat{Z}$ of $\pi$
is a union of Lagrangian submanifolds isomorphic to $\mathbb{P}^n$.
\item When $n = 1$, the exceptional locus $E$ is a tree of smooth $\mathbb{P}^1$'s
with configuration of one of the ADE singularities.
\item If $n \ge 2$, then $E$ consists of a single smooth $\mathbb{P}^n$,
and $\pi : Z \to \hat{Z}$ is analytically locally a standard resolution.
\end{enumerate}
\end{proposition}

We will also use the dimensional information about exceptional fibers in symplectic resolutions of isolated singularities in the following contexts. 

\subsection{Maximal Cohen-Macaulay sheaves}\label{MCMonsymplecticresolution}
The family of maximal Cohen-Macaulay modules is an important feature of a singularity. Suppose $R$ is a Cohen-Macaulay local ring with maximal ideal $\mathfrak{m}$. An $R$ module $M$ is called Cohen-Macaulay if $\dim(\text{Supp}(M))=\text{depth} M$. It is called \textit{maximal Cohen-Macaulay} if both of those two numbers equal $\dim R$. We have the following homological characterization (cf. \cite{Hartshorne1977}, III \S 3 Exercise 3.4)
\begin{proposition}\label{locohMCM}
Suppose $R$ is a Cohen-Macaulay local ring with maximal ideal $\mathfrak{m}$ and $M$ is a finitely generated $R$-module. Then $M$ is maximal Cohen-Macaulay if and only if $H_\mathfrak{m}^i(M)=0$ for $i< \dim R$.
\end{proposition}

For Gorenstein rings, Grothendieck local duality implies
\begin{proposition}\label{GorenMCM}
Suppose $R$ is a Gorenstein local ring and $M$ a finitely generated $R$-module. Then $M$ is maximal Cohen-Macaulay if and only if $\operatorname{Ext}^{i}_R(M, R)=0$ for all $i>0$.
\end{proposition}

For a Cohen-Macaulay variety $X$, we call a coherent sheaf $\mathcal{F}$ maximal Cohen-Macaulay if for any point $p\in X$, the stalk $\mathcal{F}_p$ is a maximal Cohen-Macaulay $\mathcal{O}_{X,p}$ module. One can easily write down the sheaf version of the previous two propositions.

To consider maximal Cohen-Macaulay sheaves on singular variety with a symplectic resolution, Proposition \ref{GorenMCM} is better since we have Grothendieck coherent duality:

\begin{lemma}\label{specsequence}
Suppose $\pi:\widetilde{X}\rightarrow X=\operatorname{Spec}\mathcal{O}$ is a symplectic resolution. For any coherent sheaf $\mathcal{F}$ on $\widetilde{X}$, there is a converging spectral sequence
$$E_2^{i,j}=\operatorname{Ext}_{\mathcal{O}}^i(R^{-j}\pi_*\mathcal{F},\mathcal{O})\Rightarrow \operatorname{Ext}_{\mathcal{\widetilde{O}}}^{i+j}(\mathcal{F}
,\mathcal{\widetilde{O}})
$$
where $\widetilde{\mathcal{O}}$ is the structure sheaf of $\widetilde{X}$.
\end{lemma}
\begin{proof}
Since this resolution is symplectic, we see by Proposition \ref{Gorat} that $X$ is Gorenstein and hence the dualizing complex of $X$ is $\mathcal{O}[n]$. On the other hand, that $\widetilde{X}$ is symplectic implies the dualizing complex of $\widetilde{X}$ is $\mathcal{\widetilde{O}}[n]$. Now $\pi^!$ pulls back dualizing complex to dualizing complex, we have $\pi^!(\mathcal{O}[n])=\mathcal{\widetilde{O}}[n]$. 
Suppose $\dim X=\dim \widetilde{X}=n$. By Grothendieck duality, we have
$$\mathbf{R}\mathcal{H}om_X(\mathbf{R}\pi_* \mathcal{F},\mathcal{O}[n])=\mathbf{R}\pi_*\mathbf{R}\mathcal{H}om_{\widetilde{X}}(\mathcal{F},\pi^!(\mathcal{O}[n]))
$$
Taking derived pushforward to a point and twisting both sides by $[-n]$, we get a global version
$$\mathbf{R}\operatorname{Hom}_X(\mathbf{R}\pi_* \mathcal{F},\mathcal{O}[n])=\mathbf{R}\operatorname{Hom}_{\widetilde{X}}(\mathcal{F},\pi^!(\mathcal{O}[n]))
$$
Finally, taking the cohomology, the Grothendieck-Serre spectral sequence gives you the result.
\end{proof}

Using this spectral sequence, one can easily conclude the following Corollary.

\begin{corollary}\label{VanishingthenMCM}
Suppose $\pi:\widetilde{X}\rightarrow X$ is a symplectic resolution. $\mathcal{F}$ is a vector bundle on $\widetilde{X}$ such that
$$R^{>0}\pi_*\mathcal{F}=R^{>0}\pi_*\mathcal{F}^\vee=0$$
then $M=\pi_*\mathcal{F}$ is maximal Cohen-Macaulay.
\end{corollary}
\begin{proof}
since the statement is local, we may assume $X=\operatorname{Spec}\mathcal{O}$ is affine. If $R^{>0}\pi_*\mathcal{F}=0$, we see that the spectral sequence in Lemma \ref{specsequence} collapses to the $0^{th}$ column, so it gives you $\text{Ext}_{\mathcal{O}}^i(M,\mathcal{O})=\text{Ext}_{\mathcal{O}}^i(\pi_*\mathcal{F},\mathcal{O})=\text{Ext}^i_\mathcal{\widetilde{O}}(\mathcal{F},\mathcal{\widetilde{O}})=H^i(\mathcal{F}^\vee)=0$. Since $X$ is affine, by Proposition \ref{GorenMCM} we see that $M$ is maximal Cohen-Macaulay.
\end{proof}

Now we state one of the main theorems in this chapter. For isolated symplectic singularities, the condition of vanishing cohomology in Corollary \ref{VanishingthenMCM} is not far from an equivalent condition for maximal Cohen-Macaulay. In order to introduce our result in a clean, elegant way, we need first the definition of a dualizing map.

Suppose $X=\operatorname{Spec}\mathcal{O}$ is an affine singular variety of dimension $2n$ and $\pi:\widetilde{X}\rightarrow X$ is a symplectic resolution. Assume further that all fibers at closed points have dimension $\leq n$. For a vector bundle $\mathcal{F}$ on $\widetilde{X}$, the spectral sequence in Lemma \ref{specsequence} has $2n$ columns (By Gorenstein property) and $n$ rows, and there is an edge morphism
$$\operatorname{Ext}_{\mathcal{O}}^{2n}(R^{n}\pi_*\mathcal{F},\mathcal{O})=E_2^{2n,-n}\rightarrow E_\infty^{2n,-n}\hookrightarrow \operatorname{Ext}_{\mathcal{\widetilde{O}}}^{n}(\mathcal{F}
,\mathcal{\widetilde{O}})=H^n(\mathcal{F}^\vee)
$$
Now for any closed point $P$ on $X$, by Proposition \ref{Gorat} we know that $\mathcal{O}_P$ is Gorenstein, so there is $\operatorname{Ext}_{\mathcal{O}_P}^{2n}(\mathbb{C},\mathcal{O}_P)=\mathbb{C}$. Then for any $\mathcal{O}_{P}$-module $M_P$, a morphism from $M_P$ to $\mathbb{C}$ induces a map $\mathbb{C}=\operatorname{Ext}_{\mathcal{O}_P}^{2n}(\mathbb{C},\mathcal{O}_P)\rightarrow \operatorname{Ext}{2n}_{\mathcal{O}_P}(M_P,\mathcal{O}_P)$, so there is a natural dualizing map $\operatorname{Hom}_{\mathcal{O}_P}(M_P,\mathbb{C})\rightarrow \operatorname{Ext}{2n}_{\mathcal{O}_P}(M_P,\mathcal{O}_P)$. Let $M=(R^{n}\pi_*\mathcal{F})_P$, we get a map 
$$\operatorname{Hom}_{\mathcal{O}_P}((R^{n}\pi_*\mathcal{F})_P,\mathbb{C})\rightarrow \operatorname{Ext}_{\mathcal{O}_P}^{2n}((R^{n}\pi_*\mathcal{F})_P,\mathcal{O}_P)
$$

When $X$ only has isolated symplectic singularities $\{P_1,P_2,...,P_N\}$, we see that $R^{n}\pi_*\mathcal{F}$ is supported only on those closed points, so $R^{n}\pi_*\mathcal{F}=\oplus_{k=1}^{N}(R^{n}\pi_*\mathcal{F})_{P_k}$. Since $\pi$ is proper, $(R^{n}\pi_*\mathcal{F})_{P_k}$ is a finite generated, hence finite length $\mathcal{O}_{P_k}$ module. Then the map above is  $(R^{n}\pi_*\mathcal{F})'_{P_k}\rightarrow \operatorname{Ext}_{\mathcal{O}_{P_k}}^{2n}((R^{n}\pi_*\mathcal{F})_{P_k},\mathcal{O}_{P_k})$ and is an isomorphism for each $k$. Here $(R^{n}\pi_*\mathcal{F})'$ is the dual vector space of $R^{n}\pi_*\mathcal{F}$. Summing up all $k$, one gets an isomorphism
$$
(R^{n}\pi_*\mathcal{F})'\xrightarrow{\sim} \operatorname{Ext}_{\mathcal{O}}^{2n}(R^{n}\pi_*\mathcal{F},\mathcal{O})
$$
Composing this isomorphism and the edge morphism above, we get a dualizing map
$$e_\mathcal{F}:(R^{n}\pi_*\mathcal{F})'\rightarrow R^n\pi_*\mathcal{F}^\vee
$$
Our theorem characterizes maximal Cohen-Macaulay using this dualizing map. This generalizes Corollary 4.3 in \cite{xu2026constructingmaximalcohenmacaulaysheaves}.

\begin{theorem}\label{isolatedsymplecticMCM}
Suppose $X$ is a $2n$-dimensional variety with only isolated singularities and $\pi:\widetilde{X}\rightarrow X$ is a symplectic resolution. For a vector bundle $\mathcal{F}$ on $\widetilde{X}$, $\pi_*\mathcal{F}$ is maximal Cohen-Macaulay if and only if
\begin{enumerate}
\item $R^i\pi_*\mathcal{F}=R^i\pi_*\mathcal{F}^\vee=0$ for $1\leq i\leq n-1$,
\item $e_{\mathcal{F}}:(R^{n}\pi_*\mathcal{F})'\rightarrow R^n\pi_*\mathcal{F}^\vee$ is an isomorphism.
\end{enumerate}
\end{theorem}
\begin{proof}
Suppose $\mathcal{O}$ and $\mathcal{\widetilde{O}}$ are the structure sheaves of $X$ and $\widetilde{X}$ respectively. Similar to what we have discussed above, suppose $\{P_1,P_2,...,P_N\}$ are all singularities of $X$, then $R^{i}\pi_*\mathcal{F}$ is supported only on those closed points for any $i>0$, so $R^{i}\pi_*\mathcal{F}=\oplus_{k=1}^{N}(R^{i}\pi_*\mathcal{F})_{P_k}$. Since $\pi$ is proper, $(R^{i}\pi_*\mathcal{F})_{P_k}$ is a finite generated, hence finite length $\mathcal{O}_{P_k}$ module. Then by Proposition \ref{Gorat}, the Gorenstein property implies $\operatorname{Ext}_{\mathcal{O}_{P_k}}^{j}((R^{i}\pi_*\mathcal{F})_{P_k},\mathcal{O}_{P_k})=0$ for $j\neq2n$ and  $\operatorname{Ext}_{\mathcal{O}_{P_k}}^{2n}((R^{i}\pi_*\mathcal{F})_{P_k},\mathcal{O}_{P_k})\cong(R^{i}\pi_*\mathcal{F})'_{P_k}$. By Proposition \ref{sympdim}, all exceptional fibers have dimension $n$, so $R^{>n}\pi_*\mathcal{F}=0$ by the Theorem on Formal Functions (cf. \cite{Hartshorne1977}, Theorem 11.1). Applying the spectral sequence in \ref{specsequence} to the local resolution $\widetilde{X}_{P_k}\rightarrow X_{P_k}=\operatorname{Spec}\mathcal{O}_{P_k}$ where $\widetilde{X}_{P_k}=\widetilde{X}\times_X X_{P_k}$, we get the following second page,
\[
\begin{tikzcd}[row sep=0.5em, column sep=0.7em]
(\pi_*\mathcal F)^\vee
  & \operatorname{Ext}_{\mathcal{O}_{P_k}}^1(\pi_*\mathcal F,\mathcal{O}_{P_k})
  & \cdots
  & \operatorname{Ext}_{\mathcal{O}_{P_k}}^{2n-1}(\pi_*\mathcal F,\mathcal{O}_{P_k})
  & \operatorname{Ext}_{\mathcal{O}_{P_k}}^{2n}(\pi_*\mathcal F,\mathcal{O}_{\mathcal{O}_{P_k}})
\\
0 & 0 & \cdots & 0 & (R^1\pi_*\mathcal F)'
\\
\vdots & \vdots & \ddots & \vdots & \vdots
\\
0 & 0 & \cdots & 0 & (R^n\pi_*\mathcal F)'
\end{tikzcd}
\]

One can check that for any $2\leq m \leq n+1$, the only nonzero $m$-step differential $d^{m}$ is from $\operatorname{Ext}^{2n-m}_{\mathcal{O}_{P_k}}(\pi_*\mathcal F,\mathcal{O}_{P_k})$ to $(R^{m-1}\pi_*\mathcal{F})'$ and that the spectral sequence stabilizes after $n+1$ steps. The pattern of the differentials is like
\[
\begin{tikzcd}[row sep=1em, column sep=1em]
\ast & \ast & \cdots & \ast\arrow[rrrrrdddd,"d_{n+1}"] & \cdots\arrow[rrrrddd,"\cdots"]& \ast\arrow[rrrdd,"d_3"] & \ast\arrow[rrd,"d_2"] & \ast & \ast \\
0 & 0 & \cdots & 0 & \cdots &   &   &   & \ast \\
0 & 0 & \cdots & 0 & \cdots &   &   &   & \ast \\
\vdots & \vdots & \ddots & \vdots & \ddots & \vdots & \vdots & \vdots & \vdots\\
0 & 0 & \cdots & 0 & \cdots & 0 & 0 & 0 & \ast
\end{tikzcd}
\]

For $2\leq m\leq n+1$, the differentials give us exact sequences:
$$0\rightarrow \ker \,d^m\rightarrow \operatorname{Ext}^{2n-m}_{\mathcal{O}_{P_k}}(\pi_*\mathcal F,\mathcal{O}_{P_k})\rightarrow \text{im} \,d^m\rightarrow 0
$$
Denote the structure sheaf on $\widetilde{X}_{P_k}$ by $\widetilde{\mathcal{O}}_{P_k}$, we have an identification of $\operatorname{Ext}_{\widetilde{\mathcal{O}}_{P_k}}^{2n-m+1}(\mathcal{F},\widetilde{\mathcal{O}}_{P_k})$ with $H^{2n-m+1}(\mathcal{F}^\vee,{\widetilde{X}_{P_k}})=(R^{2n-m+1}\pi_*\mathcal{F}^\vee)_{P_k}$ . For $3\leq m\leq n+1$, the convergence in \ref{specsequence} gives us exact sequences 
$$0\rightarrow(R^{m-1}\pi_*\mathcal{F})'_{P_k}/\text{im} \,d^m\rightarrow (R^{2n-m+1}\pi_*\mathcal{F}^\vee)_{P_k}\rightarrow \ker d^{m-1}\rightarrow 0
$$
and for $m=2$, a single short exact sequence
$$
0\rightarrow(R^{1}\pi_*\mathcal{F})'_{P_k}/\text{im} \,d^2\rightarrow (R^{2n-1}\pi_*\mathcal{F}^\vee)_{P_k}\rightarrow\operatorname{Ext}^{2n-1}_{\mathcal{O}_{P_k}}(\pi_*\mathcal F,\mathcal{O}_{P_k}) \rightarrow 0
$$
One also gets $\ker d^{n+1}=(R^{n-1}\pi_*\mathcal{F}^\vee)_{P_k}$ considering the $i+j=n-1$ diagonal. 
Now one can actually integrate these short exact sequences into a long exact sequence
\[
\begin{tikzcd}[row sep=1.3em, column sep=2.4em]
0 \arrow[r]
& (R^{n-1}\pi_*\mathcal{F}^\vee)_{P_k} \arrow[r]
& \operatorname{Ext}^{n-1}_{\mathcal{O}_{P_k}}(\pi_*\mathcal F,\mathcal{O}_{P_k}) \arrow[r,"d^{n+1}"]
& (R^{n}\pi_*\mathcal{F})' \arrow[r,"(e_{\mathcal{F}})_{P_k}"]
& (R^{n}\pi_*\mathcal{F}^\vee)_{P_k}\arrow[lld,out=0,in=180,looseness=1.1]
\\
& 
& \operatorname{Ext}^{n}_{\mathcal{O}_{P_k}}(\pi_*\mathcal F,\mathcal{O}_{P_k}) \arrow[r,"d^{n}"]
& (R^{n-1}\pi_*\mathcal{F})' \arrow[r]
& (R^{n+1}\pi_*\mathcal{F}^\vee)_{P_k}\arrow[lld,out=0,in=180,looseness=1.1]
\\
&
&
\quad\quad\qquad\cdots\qquad\quad \arrow[r]
& \cdots \arrow[r]
& \quad\qquad\cdots\qquad\ \arrow[lld,out=0,in=180,looseness=1.1]
\\
&
&
\operatorname{Ext}^{2n-2}_{\mathcal{O}_{P_k}}(\pi_*\mathcal F,\mathcal{O}_{P_k}) \arrow[r,"d^2"]
& (R^{1}\pi_*\mathcal{F})' \arrow[r]
& (R^{2n-1}\pi_*\mathcal{F}^\vee)_{P_k}\arrow[lld,out=0,in=180,looseness=1.1]
\\
&
&
\operatorname{Ext}^{2n-1}_{\mathcal{O}_{P_k}}(\pi_*\mathcal F,\mathcal{O}_{P_k}) \arrow[r]
& 0
&
\end{tikzcd}
\]
We also have 
$\operatorname{Ext}^{2n}_{\mathcal{O}_{P_k}}(\pi_*\mathcal F,\mathcal{O}_{P_k})\cong (R^{2n}\pi_*\mathcal{F}^\vee)_{P_k}$
and $\operatorname{Ext}^{i}_{\mathcal{O}_{P_k}}(\pi_*\mathcal F,\mathcal{O}_{P_k})\cong (R^{i}\pi_*\mathcal{F}^\vee)_{P_k}$ for $1\leq i\leq n-2$ from that convergence of spectral sequence. Again by the Theorem on Formal Functions, $R^{>n}\pi_*\mathcal{F}^\vee=0$, so $\operatorname{Ext}^{2n}_{\mathcal{O}_{P_k}}(\pi_*\mathcal F,\mathcal{O}_{P_k})=\operatorname{Ext}^{2n-1}_{\mathcal{O}_{P_k}}(\pi_*\mathcal F,\mathcal{O}_{P_k})=0$ and $d^m:\operatorname{Ext}^{2n-m}_{\mathcal{O}_{P_k}}(\pi_*\mathcal F,\mathcal{O}_{P_k})\rightarrow(R^{m-1}\pi_*\mathcal{F})'$ are isomorphisms for $3\leq m\leq n-1$.

By Proposition \ref{GorenMCM}, $\pi_*\mathcal{F}$ is maximal Cohen-Macaulay if and only if $\mathcal{E}xt^{>0}_{\mathcal{O}}(\pi_*\mathcal F,\mathcal{O})=0$. Now $\mathcal{F}$ is a vector bundle on smooth locus, so these $\mathcal{E}xt$ sheaves only support on $P_1, P_2,...,P_n$. Hence the maximal Cohen-Macaulay condition is also equivalent to $\operatorname{Ext}^{>0}_{\mathcal{O}_{P_k}}(\pi_*\mathcal F,\mathcal{O}_{P_k})=0$ for any $k$. Now assume this holds, we see that $(R^{i}\pi_*\mathcal{F}^\vee)_{P_k}=\operatorname{Ext}^{i}_{\mathcal{O}_{P_k}}(\pi_*\mathcal F,\mathcal{O}_{P_k})=0$ for $1\leq i\leq n-2$. From the long exact sequence, we see that $(R^{n-1}\pi_*\mathcal{F}^\vee)_{P_k}=0$, $(R^{i}\pi_*\mathcal{F})_{P_k}$ for $1\leq i\leq n-1$ and $(e_{\mathcal{F}})_{P_k}:(R^{n}\pi_*\mathcal{F})'_{P_k}\rightarrow(R^{n}\pi_*\mathcal{F}^\vee)_{P_k}$
is an isomorphism. This works for all $P_k$, hence (1), (2) hold.

Conversely, suppose (1), (2) hold, then $\operatorname{Ext}^{i}_{\mathcal{O}_{P_k}}(\pi_*\mathcal F,\mathcal{O}_{P_k})\cong (R^{i}\pi_*\mathcal{F}^\vee)_{P_k}=0$ for $1\leq i\leq n-2$, $\operatorname{Ext}^{n-1}_{\mathcal{O}_{P_k}}(\pi_*\mathcal F,\mathcal{O}_{P_k})=\ker (e_{\mathcal{F}})_{P_k}=0$, $\operatorname{Ext}^{n}_{\mathcal{O}_{P_k}}(\pi_*\mathcal F,\mathcal{O}_{P_k})=\operatorname{coker} (e_{\mathcal{F}})_{P_k}=0$, and $\operatorname{Ext}^{i}_{\mathcal{O}_{P_k}}(\pi_*\mathcal F,\mathcal{O}_{P_k})=(R^{2n-1-i}\pi_*\mathcal{F})'_{P_k}=0$ for $n\leq i+1\leq 2n-2$. Since we already have $\operatorname{Ext}^{2n}_{\mathcal{O}_{P_k}}(\pi_*\mathcal F,\mathcal{O}_{P_k})=\operatorname{Ext}^{2n-1}_{\mathcal{O}_{P_k}}(\pi_*\mathcal F,\mathcal{O}_{P_k})=0$, we conclude that $\operatorname{Ext}^{>0}_{\mathcal{O}_{P_k}}(\pi_*\mathcal F,\mathcal{O}_{P_k})=0$, so $\pi_*\mathcal{F}$ is maximal Cohen-Macaulay. 

In summary, we have proven that $\pi_*\mathcal{F}$ is maximal Cohen-Macaulay if and only if (1), (2) hold.
\end{proof}
\begin{corollary}\label{affinevectorMCM}
Suppose $X$ is an affine $2n$-dimensional variety with only isolated singularities and $\pi:\widetilde{X}\rightarrow X$ is a symplectic resolution. If $\mathcal{F}$ is a vector bundle on $\widetilde{X}$ such that $\pi_*\mathcal{F}$ is maximal Cohen-Macaulay, then $H^i(\mathcal{F})=H^{i}(\mathcal{F}^\vee)=0$ for $1\leq i\leq n-2$ and $\dim_{\mathbb{C}}H^n(\mathcal{F})=\dim_{\mathbb{C}}H^n(\mathcal{F}^\vee)$.
\end{corollary}

This corollary and Corollary \ref{VanishingthenMCM} above are extremely useful for us to determine maximal Cohen-Macaulay divisors on symplectic resolutions, as we will show in the following contexts.

For isolated singularities on affine varieties, It's also a good way to determine maximal Cohen-Macaulay sheaves using local cohomology. Recall that a coherent sheaf is called \textit{reflexive} if the natural map $\mathcal{F}\rightarrow\mathcal{F}^{\vee\vee}$ is an isomorphism. There are well-known properties for reflexivity: For a reflexive sheaf $\mathcal{F}$ on a normal Cohen-Macaulay scheme $X$, we have
\begin{enumerate}
\item $F$ is torsion free. Furthermore, $\mathcal{F}$ is reflexive if and only if $\mathcal{F}$ satisfies the $S^2$ condition, i.e., $\operatorname{depth}\mathcal{F}_P\geq\min\{2,\operatorname{codim \overline{P}}\}$ for any point $P\in X$. So one can see maximal Cohen-Macaulay sheaves are reflexive when $\dim X\geq 2$.
\item For any closed subscheme $Y$ with $\operatorname{codim}_X Y\geq 2$, we have $\mathcal{F}=i_*(\mathcal{F}|_{X\setminus Y})$, where $i: X\setminus Y\hookrightarrow X$ is the open embedding. 
\item For any coherent sheaf $\mathcal{G}$, its dual $\mathcal{F}=\mathcal{G}^\vee$ is reflexive. 
\item Reflexive sheaves of rank 1 are in one-to-one correspondence with the Weil divisors on $X$.
\end{enumerate}

\begin{proposition}\label{localseqMCM}
Suppose $X$ is a Cohen-Macaulay affine variety with dimension $n\geq 2$ and $P\in X$ is a closed point. For a reflexive sheaf $\mathcal{F}$ on $X$, $\mathcal{F}$ is maximal Cohen-Macaulay if and only if 
$H^i(\mathcal{F},X\setminus P)=0$ for $1\leq i\leq n-2$.
\end{proposition}
\begin{proof}
Since $X$ is affine, we have $H^{>0}(X,\mathcal{F})=0. $ Consider the long exact sequence (cf. \cite{Hartshorne1977} III \S 2, Exercise 2.3 (e)):
\[
\begin{tikzcd}[column sep=1.4em, row sep=1em]
0 \arrow[r]
& H_P^0(X,\mathcal F) \arrow[r]
& H^0(X,\mathcal F) \arrow[r]
& H^0(X\setminus P,\mathcal F) \arrow[lld,out=0,in=180,looseness=1.2]& \ \\
& H_P^1(X,\mathcal F) \arrow[r]
& 0 \arrow[r]
& H^1(X\setminus P,\mathcal F) \arrow[lld,out=0,in=180,looseness=1.2]& \ \\
& H_P^2(X,\mathcal F) \arrow[r]
& \cdots &\ \\
&\  & \cdots&\qquad\qquad\qquad \arrow[lld,out=0,in=180,looseness=1.2]& \  \\
& H_P^n(X,\mathcal F) \arrow[r]
& 0 \arrow[r]
& H^n(X\setminus P,\mathcal F) \arrow[r] & 0
\end{tikzcd}
\]
By Proposition \ref{locohMCM}, $\mathcal{F}$ is maximal Cohen-Macaulay at $P$ if and only if $H^{<n}_P(X,\mathcal{F})=0$. Now $F$ is reflexive, so $\mathcal{F}=i_*(\mathcal{F}|_{X\setminus P})$, where $i: X\setminus P\hookrightarrow X$ is the open embedding. Since $X$ is affine, we conclude that  $H^0(X,\mathcal{F})\rightarrow H^0(X\setminus P,\mathcal{F})$ is an isomorphism, then $H_P^0(X,\mathcal F)=H_P^1(X,\mathcal F)=0$ and $H_P^i(X,\mathcal F)\cong H^{i-1}(X\setminus P,\mathcal{F})$ for $2\leq i\leq n-1$ by the long exact sequence above. Hence $\mathcal{F}$ is maximal Cohen-Macaulay if and only if $H^i(\mathcal{F},X\setminus P)=0$ for $1\leq i\leq n-2$.
\end{proof}

\subsection{Maximal Cohen-Macaulay divisors on $\mathcal{N}_{n+1,1}$}\label{MCMdivisoronA_n}
We now check our theory on the rank 1 reflexive sheaves with respect to the Weil divisors on $\mathcal{N}_{n+1,1}$, the minimal nilpotent orbit in $sl_{n+1}$ parameterizing nilpotent $(n+1)\times(n+1)$ matrices of rank $\leq 1$. We assume $n\geq 2$ since $n=1$ case is completely classified in \cite{ArtinVerdierEsnault1985}.

From the constructions in section \ref{miniorbit}, we know $\mathcal{N}_{n+1}$ admits a symplectic resolution $\pi: T^*\mathbb{P}^n\rightarrow\mathcal{N}_{n+1,1}$. So for $n\geq 2$, we have the following relations of Picard groups
$$\operatorname{Cl}\mathcal{N}_{n+1}=\operatorname{Cl}(\mathcal{N}_{n+1}\setminus 0)=\operatorname{Cl}(T^*\mathbb{P}^n\setminus s(\mathbb{P}^n))=\operatorname{Cl}(T^*\mathbb{P}^n)=\operatorname{Pic}\mathbb{P}^n=\mathbb{Z}
$$
where $s$ is the zero section. Now for a reflexive sheaf $\mathcal{F}$ on $\mathcal{N}_{n+1,1}$, define $\overline{\mathcal{F}}=(\pi^*\mathcal{F})^\vee$. From \cite{xu2026constructingmaximalcohenmacaulaysheaves} Lemma 3.4, we see that $\pi_*\overline{\mathcal{F}}=\mathcal{F}$. So if we suppose $D$ is the (positive) generator of $\operatorname{Pic}\mathcal{N}_{n+1}$, we have $(\pi^*\mathcal{O}(mD))^{\vee\vee}=\overline{\mathcal{O}(mD)}=\mathcal{O}_{T^*\mathbb{P}^n}(1)$ and $\pi_*\mathcal{O}_{T^*\mathbb{P}^n}(m)=\mathcal{O}(mD)$, where $\mathcal{O}_{T^*\mathbb{P}^n}(m)$ is the pullback of $\mathcal{O}_{\mathbb{P}^n}(m)$ along the projection. By Theorem \ref{isolatedsymplecticMCM} (or Corollary \ref{affinevectorMCM}), if $\mathcal{O}(mD)$ is maximal Cohen-Macaulay, we should have $H^i(\mathcal{O}_{T^*\mathbb{P}^n}(m))=H^1(\mathcal{O}_{T^*\mathbb{P}^n}(-m))=0$ for $1\leq i\leq n-1$ and $\dim_\mathbb{C} H^n(\mathcal{O}_{T^*\mathbb{P}^n}(m))=\dim_\mathbb{C}H^n(\mathcal{O}_{T^*\mathbb{P}^n}(-m))$.  Now we have the following computational lemma:
\begin{lemma}\label{MCMWD} For any integer $m$, we have 
\begin{enumerate}
\item $H^i(\mathcal{O}_{T^*\mathbb{P}^n}(m))=0$ for $1\leq i \leq n-2$,
\item $H^{n-1}(\mathcal{O}_{T^*\mathbb{P}^n}(m))=0$ if and only if $m\geq -n$, and in this case, we also have $H^n(\mathcal{O}_{T^*\mathbb{P}^n}(m))=0$.
\end{enumerate}
\end{lemma}
\begin{proof}
Firstly, there is a graded decomposition 
$$
H^i(\mathcal{O}_{T^*\mathbb{P}^2}(m))=H^i(\oplus_{j=0}^{\infty}S^j(T_{\mathbb{P}^2})(m))=\oplus_{j=0}^{\infty}H^i(S^j(T_{\mathbb{P}^2})(m))
$$
Now we compute each degree. Taking symmetric product of Euler sequence and twisting by $\mathcal{O}_{\mathbb{P}^2}(m)$, we get a short exact sequence:
$$
0\rightarrow\mathcal{O}_{\mathbb{P}^n}(j-1+m)^{\binom{j+n-1}{n}}\rightarrow\mathcal{O}_{\mathbb{P}^n}(j+m)^{\binom{j+n}{n}}\rightarrow S^j(T_{\mathbb{P}^n})(m)\rightarrow 0
$$
Take the long exact sequence of cohomology and use the vanishing of the middle cohomology for line bundles, we see easily that $H^i(S^j(T_{\mathbb{P}^n})(m))=0$ for any $1\leq i\leq n-2$. We also have the following exact sequence for $H^{n-1}$ and $H^n$,
$$0\rightarrow H^{n-1}(S^j(T_{\mathbb{P}^n})(m))\rightarrow H^n(\mathcal{O}_{\mathbb{P}^n}(j-1+m))^{\binom{j+n-1}{n}}\rightarrow H^n(\mathcal{O}_{\mathbb{P}^n}(j+m))^{\binom{j+n-1}{n}}\rightarrow H^n(S^j(T_{\mathbb{P}^n})(m))\rightarrow 0
$$
If $H^{n-1}(\mathcal{O}_{T^*\mathbb{P}^2}(m))=0$, we see that $H^{n-1}(S^j(T_{\mathbb{P}^n})(m))=0$ for all $j\geq 0$. Suppose conversely $m\leq-n-1$, then $-n-m\geq 1$. Now let $j=-n-m$, we have $H^n(\mathcal{O}_{\mathbb{P}^n}(j+m))^{\binom{j+n}{n}}=H^n(\mathcal{O}_{\mathbb{P}^n}(-n))^{\binom{-m}{n}}=0$ while $H^n(\mathcal{O}_{\mathbb{P}^n}(j-1+m))^{\binom{j+n-1}{n}}=H^n(\mathcal{O}_{\mathbb{P}^n}(-n-1))^{\binom{-m-1}{n}}=\mathbb{C}^{\binom{-m-1}{n}}$ where $-m-1\geq n+1-1=n$, so $\binom{-m-1}{n}\geq 1$. Hence $H^{n-1}(S^{-n-m}(T_{\mathbb{P}^n})(m))=H^n(\mathcal{O}_{\mathbb{P}^n}(-n-1))^{\binom{-m-1}{n}}=\mathbb{C}^{\binom{-m-1}{n}}\neq 0$, a contradiction. So we must have $m\geq -n$.

When $m\geq -n$, we see that $j-1+m\geq -n$ for all $j\geq 1$, so $H^{n-1}(S^j(T_{\mathbb{P}^n})(m))=0$ for any $j\geq 1$. For $j=0$, this is also $0$ since $\binom{j+n-1}{n}=\binom{n-1}{n}=0$, so $H^{n-1}(\mathcal{O}_{T^*\mathbb{P}^2}(m))=0$. Finally, in this case $j+m\geq 0-n=-n$, so $H^n(\mathcal{O}_{\mathbb{P}^n}(j+m))^{\binom{j+n-1}{n}}=0$, which implies $H^n(S^j(T_{\mathbb{P}^n})(m))=0$. Therefore, $H^{n}(\mathcal{O}_{T^*\mathbb{P}^2}(m))=0$.
\end{proof}
This lemma and Theorem \ref{isolatedsymplecticMCM} imply,
\begin{corollary}\label{A_nMCMdivisors}
The maximal Cohen-Macaulay Weil divisors on $\mathcal{N}_{n+1,1}$ are exactly the divisors corresponding to  reflexive sheaves $\mathcal{O}(mD)$ for $-n\leq m\leq n$, where $D$ is a generator for the Weil divisor group of $\mathcal{N}_{n+1,1}$.
\end{corollary}
\begin{proof}
Since $\mathcal{O}(mD)=\pi_*\mathcal{O}_{T^*\mathbb{P}^n}(m)$, it suffices to check the conditions in Theorem \ref{isolatedsymplecticMCM}. Firstly $R^1\pi_*\mathcal{O}_{T^*\mathbb{P}^n}(m)=R^1\pi_*\mathcal{O}_{T^*\mathbb{P}^n}(-m)=0$ implies $m\geq -n$, $-m\geq -n$, so $-n\leq m\leq n$. While in this case, by Lemma \ref{MCMWD} above, $R^2\pi_*\mathcal{O}_{T^*\mathbb{P}^n}(m)=R^2\pi_*\mathcal{O}_{T^*\mathbb{P}^n}(-m)=0$, so $e_{\mathcal{O}_{T^*\mathbb{P}^n}(m)}$ is a trivial isomorphism. Then we see that $\mathcal{O}(mD)=\pi_*\mathcal{O}_{T^*\mathbb{P}^n}(m)$ are all maximal Cohen-Macaulay Weil divisors.
\end{proof}
\section{Equivariant Maximal Cohen-Macaulay sheaves on minimal orbit closure of $\mathfrak{sl}_{n+1}$}
In \cite{xu2026constructingmaximalcohenmacaulaysheaves}, we construct some maximal Cohen-Macaulay sheaves on $\mathcal{N}_{3,1}$, the variety of $3\times 3$ nilpotent matrices with rank $\leq1$, but little classification work is done here. One can see the reason why the classification is so difficult: even for rank 2 sheaves, there are indecomposable maximal Cohen-Macaulay sheaves $\pi_*f^*V$ where $\pi:T^*\mathbb{P}^2\rightarrow \mathcal{N}_{3,1}$ is the symplectic resolution, $f:T^*\mathbb{P}^2\rightarrow \mathbb{P}^2$ is the projection map and $V$ is a vector bundle on $\mathbb{P}^2$ belonging to one of the following families
\begin{enumerate}
\item $T_{\mathbb{P}^2}(-1)$, $T_{\mathbb{P}^2}(-2)$.
\item Nontrivial extension of $I_x$ by $\mathcal{O}_{\mathbb{P}^2}$, where $I_x$ is the ideal sheaf of a closed point $x$ on $\mathbb{P}^2$.
\item Stable 2-bundles with $c_1=0$ and $c_2=1$.
\end{enumerate}
One can also construct indecomposable maximal Cohen-Macaulay modules of any rank $r\geq 2$ using Steiner bundles on $\mathbb{P}^2$. In this way one can also construct maximal Cohen-Macaulay sheaves on $\mathcal{N}_{n+1,1}$ with various rank. Details can be found in \cite{xu2026constructingmaximalcohenmacaulaysheaves}.

These constructions imply the high complexity of the classification problem of maximal Cohen-Macaulay sheaves on higher dimensional symplectic singularities. However, if we restrict our attention to those sheaves having a compatible $SL(n+1)$ action with the conjugation $SL(n+1)$ action on $\mathcal{N}_{n+1,1}$, namely the \textit{$SL(n+1)$-equivariant sheaves}, we will see that when the corresponding representation of the stabilizer subgroup is irreducible, the Cohen-Macaulay condition can be fully described using the characters of that representation. 
\subsection{Equivariant Sheaves and equivariant divisors}
We first introduce definition and basic properties of the objects we are going to study.
\begin{definition}
Let an algebraic group $G$ act on a scheme $X$, with action map $a:G\times X\to X$ and projection $p_2:G\times X\to X$. A \textit{$G$-equivariant coherent sheaf} on $X$ is a coherent sheaf $\mathcal{F}$ together with an isomorphism
$$
\phi:a^*\mathcal{F}\xrightarrow{\sim} p_2^*\mathcal{F}
$$
such that on $G\times G\times X$ one has
$$
(\mu\times \mathrm{id}_X)^*\phi
=
(\mathrm{id}_G\times p_2)^*\phi\circ (\mathrm{id}_G\times a)^*\phi,
$$
where $\mu:G\times G\to G$ is the multiplication map.
\end{definition}

For an algebraic group $G$ acting on a normal scheme $X$, we call a Weil divisor on $X$ \textit{$G$-equivariant} if the corresponding reflexive sheaf admits a $G$-equivariant structure.

As sheaves, one can find locally free resolutions for equivariant sheaves, as the following theorem from \cite{ChrissGinzburg1997} implies.
\begin{proposition}[\cite{ChrissGinzburg1997} Proposition 5.1.26]
Let $X$ be a normal quasi-projective $G$-variety and $\mathcal{L}$ be a very ample $G$-equivariant Cartier divisor.
Then any $G$-equivariant coherent sheaf $\mathcal{F}$ on $X$ is a quotient of some
$G$-equivariant locally free sheaf of the form $V\otimes(\mathcal{L}^\vee)^{\otimes n}$ for some finite dimensional $G$-representation $V$ and integer $n$.
\end{proposition}

For an affine variety, the structure sheaf is very ample, so we can perform free resolutions of coherent sheaves. Then the higher syzygies are maximal Cohen-Macaulay $G$-equivariant sheaves. As we have a primary interest in classifying/constructing natural objects appearing in free resolutions, it worth studying  maximal Cohen-Macaulay $G$-equivariant sheaves in details.

Now for equivariant sheaves on homogeneous space, we have the following result.

\begin{proposition}\label{GEQUITHENVECT}
Suppose $G$ is an algebraic group and $H$ is its algebraic subgroup. The $G$-equivariant sheaves on $G/H$ are finite dimensional vector bundles and are one-to-one correspond to finite dimensional representations of $H$.
\end{proposition}
\begin{proof}
The vector bundle statement follows from the generic local freeness and transitivity of the $G$-action. Given an $H$-representation $V$, denote its dual representation by $V^]\vee$, one can define a $G$-equivariant bundle on $G/H$ by 
$$G\times_H V^\vee=\{(g,f)|g\in G, f\in V^\vee\}/\sim$$
where $(g,f)\sim(gh^{-1},h.f)=(gh^{-1},f\circ h^{-1})$ for all $h\in H$. Conversely, for a $G$-equivariant vector bundle on $G/H$, one obtains a $H$ representation by taking the dual space of the fiber at $[e]$, the class of the identity element.
\end{proof}

The Borel--Weil--Bott theorem is the most useful tool computing cohomology on the flag varieties, which is the homogeneous space of $G$ quotient by a Borel subgroup. This relation actually exists for all parabolic groups. The following version of BWB theorem is cited from \cite{Akhiezer} \S 4.3.

\begin{theorem}[Borel--Weil--Bott for $G/P$]\label{BWBparabolic}Suppose $G$ is an complex semisimple group and $P$ is its parabolic subgroup.
Let $\mathcal{E}_{\lambda}=G\times_{P}V_\lambda^\vee$ be the $G$-equivariant vector bundle
over $Y=G/P$ corresponding to an irreducible representation
$P$ of highest weight $\lambda$.

\begin{enumerate}
\item[(i)] If $\lambda+\rho$ is singular, then $H^q(Y,\mathcal{E}_\lambda)=0$ for all $q\ge 0$.
\item[(ii)] If $\lambda+\rho$ is regular, then $H^q(Y,\mathcal{E}_\lambda)$ is nonzero if and only if $q=q_{\lambda}$,  in which case it is the dual of the irreducible representation of $G$ of highest weight $w(\lambda+\rho)-\rho$, where $w$ is the unique element of $W$ that makes $w(\lambda+\rho)$ dominant.
\end{enumerate}
\end{theorem}

However, our minimal orbits $\mathcal{O}_{\min}$ are not proper quotients of the corresponding group, so there should be some justification for computing the cohomology of equivariant sheaves there. 

\subsection{Equivariant divisors on minimal orbit closure of $\mathfrak{sl}_3$}
Let's start with the simplest case $G=SL(3)$ and $\overline{\mathcal{O}_{\min}}=\mathcal{N}_{3,1}$. In this case, for a $G$-equivariant maximal Cohen-Macaulay sheaf $\mathcal{F}$, by the property of reflexive sheaves we see $\mathcal{F}=i_*\mathcal{F}|_{\mathcal{O}_{\min}}$, where $i:\mathcal{O}_{\min}=\mathcal{N}_{3,1}\setminus 0\rightarrow\mathcal{N}_{3,1}$ is the open embedding. Now we know $\mathcal{O}_{\min}$ is a homogeneous space, so by Proposition \ref{GEQUITHENVECT}, $\mathcal{F}|_{\mathcal{O}_{\min}}$ is a vector bundle and hence obviously maximal Cohen-Macaulay, so it suffices to check the maximal Cohen-Macaulay condition at the singular point $0$. Now by Proposition \ref{localseqMCM}, $\mathcal{F}$ is maximal Cohen-Macaulay at $0$ if and only if
$$H^1(\mathcal{F|_{\mathcal{O}_{\min}}})=H^2(\mathcal{F|_{\mathcal{O}_{\min}}})=0
$$
In order to compute the cohomology, we analyze the orbit structure carefully: all $3\times 3$ nilpotent matrices of rank 1 are conjugate to 
$$\begin{bmatrix}
    0 & 0 & 1\\
    0 & 0 & 0\\
    0 & 0 & 0
\end{bmatrix}
$$
and the stabilizer of this element in $SL(3)$ is 
$$H=\left\{\left.\begin{bmatrix}
    a & b & c\\
    0 & \frac{1}{a^2} & d\\
    0 & 0 & a
\end{bmatrix}\right | a,b,c,d\in\mathbb{C}, \ a\neq 0\right\}
$$
Again by \ref{GEQUITHENVECT}, the vector bundles on $\mathcal{O}_{\min}=G/H$ are in one-to-one correspondence with $H$-representations. We are mainly interested in those \textit{irreducible} representations of $H$. Since $H$ is solvable in this case, all irreducible representations $V$ of $H$ are 1-dimensional, so they correspond to divisors on $\mathcal{O}_{\min}$, hence Weil divisors on $\mathcal{N}_{3,1}$.

We know how to compute maximal Cohen-Macaulay divisors in the previous section. The question is, if there exists a representation theoretic way to compute this. The answer is positive. One needs to notice that $H$ is "one step from parabolic". We denote by $S$ the subgroup in $SL(3)$ of the form
$$S=\left\{\left.\begin{bmatrix}
    x & 0 & 0\\
    0 & 1 & 0\\
    0 & 0 & \frac{1}{x}
\end{bmatrix}\right | x\in\mathbb{C}^\times\right\}
$$
One sees that $S\cong\mathbb{G}_m$, $H\cap S=\{\operatorname{diag}(1,1,1),\operatorname{diag}(-1,1,-1)\}\cong\mu_2$ and $S$ normalizes $H$. It's also easy to see that $H$ and $S$ generate the Borel subgroup $P$ of $SL(3)$ consisting of all upper triangular matrices, so we have 
$$P=(S\ltimes H)/\mu_2
$$
There is a natural quotient map $q:G/H\rightarrow G/P$. The fiber at the identity class $[1]$ is $P/H$. The injection $S\hookrightarrow P$ induces an isomorphism $S/\mu_2\cong P/H$. Now $S/\mu_2\cong\mathbb{G}_m/\mu_2\cong \mathbb{G}_m$, so $q$ is actually an affine map, and we can compute the cohomology on $G/H$ by pushing forward to $G/P$:
Given any equivariant line bundle $\mathcal{E}_V$ that corresponds to some irreducible 1-dimensional representation $\rho: H\rightarrow GL(V)$ of $H$, we have
$$H^k(\mathcal{E}_V,G/H)=H^{k}(q_*\mathcal{E}_V,G/P)
$$
Now we need to figure out what the $P$-representation corresponding to the equivariant bundle $q_*\mathcal{E}_V$ is. By the base change theorem, we have
$$(q_*\mathcal{E}_V)_{[1]}=H^0(\mathcal{E}_V,P/H)=H^0(\mathcal{E}_V,S/\mu_2)
$$
where on the closed subscheme $P/H\cong S/\mu_2$, we also identify $\mathcal{E}_V$ as the equivariant bundle of the restriction $\rho|_{\mu_2}:\mu_2\rightarrow GL(V)$. Denote by $\operatorname{Mor}(S,V^\vee)$ the algebraic maps from $S$ to $V^\vee$ and $\operatorname{Mor}(S,V^\vee)_{\mu_2}$ the "balancing" subset consisting of functions $f$ satisfying $(h.f)(g)=\rho^\vee(h)^{-1}f(gh)$ for $h\in\mu_2$ and $g\in S$. There is a natural identification
$$
H^0(\mathcal{E}_V,S/\mu_2)=\operatorname{Mor(S,V)}^{\mu_2}
$$
Then if we write $S=\operatorname{Spec}\mathbb{C}[t,t^{-1}]$, $\operatorname{Mor}(S,V)$ can be identified with $\mathbb{C}[t,t^{-1}]\otimes V^\vee$, where $(t^{j}\otimes \phi)(x)=x^j\phi$ for $x\in S=\mathbb{G}_m$ and $\phi\in V^\vee$. By computing $\mu_2$ balancing elements, one sees that 
$$H^0(\mathcal{E}_V,S/\mu_2)=\begin{cases}
    \displaystyle \bigoplus_{i\in\mathbb{Z}} \mathbb{C}t^{2i}\otimes V^\vee, & \rho|_{\mu_2}\  \text{is trivial}\\
    \displaystyle \bigoplus_{i\in\mathbb{Z}} \mathbb{C}t^{2i+1}\otimes V^\vee, & \rho|_{\mu_2} \ \text{is nontrivial}
\end{cases}
$$
While we know this cohomology group is also $H^0(\mathcal{E}_V,P/H)$, which can be identified with $\operatorname{Mor}(P,V^\vee)^H$. Then $t^j\otimes \phi$ can be regarded as an balancing function from $P$ to $V^\vee$ mapping $yh\in P$ to $y^j\rho^\vee(h)^{-1}\phi$. The natural $P$ action becomes
$$(yh.(t^j\otimes \phi))(x)=(t^j\otimes \phi)(h^{-1}y^{-1}x)=(y^{-1}x)^j\rho^\vee((y^{-1}x)^{-1}h(y^{-1}x))\phi
$$
Since $V$ is irreducible, the unipotent of $H$ acts trivially, so $\rho$ factors as $D=H/U(H)\rightarrow GL(V)$, where $D$ is also the diagonal group $\{\operatorname{diag}(a,a^{-2},a)|a\in\mathbb{C}^\times\}$. Then by commutativity of diagonal elements, one sees that $\rho^\vee((y^{-1}x)^{-1}h(y^{-1}x))=\rho^\vee(h)$ for all $x,y\in S$ and $h\in H$. Then the $P$-action is simply $(yh.(t^j\otimes v))(x)=y^{-j}x^j\rho^\vee(h)v$, so $yh.(t^j\otimes v)=t^j\otimes y^{-j}\rho^\vee(h) v$. Then the splitting above is actually a splitting of $P$-representations, i.e., each summand $\mathbb{C}t^{j}\otimes V^\vee$ is $P$-stable.

Now we discuss the weight of each summand in $H^0(\mathcal{E}_{V},P/H)=H^0(\mathcal{E}_V,S/\mu_2)$. Suppose $V=V(\lambda)$ is given by $H$-character $\lambda: H\rightarrow\mathbb{C}^\times$ sending $\operatorname{diag}(a,a^{-2},a)$ to $a^\lambda$. Then the $P$ action on $\mathbb{C}t^{-j}\otimes V^\vee$ is given by
$$\begin{bmatrix}
    ay & b & c\\
    0 & \frac{1}{a^2} & d\\
    0 & 0 & \frac{a}{y}
\end{bmatrix}.(t^j\otimes v^*)=y^{j}a^{-\lambda} (t^j\otimes v^*)=(ay)^{-\frac{\lambda-j}{2}}\left(\frac{a}{y}\right)^{-\frac{\lambda+j}{2}}(t^j\otimes v^*)
$$
One should notice that $\frac{\lambda\pm j}{2}$ are both integers because $\rho_{\mu_2}$ is nontrivial if and only if $\lambda$ is odd. Therefore, $\mathbb{C}t^{-j}\otimes V^\vee$ is an irreducible $P$-representation with weight $(\frac{\lambda-j}{2},0,\frac{\lambda+j}{2})$ in $P$'s weight lattice $\mathbb{Z}^3/\mathbb{Z}(1,1,1)$. Therefore, as a $P$-representation, we have
$$(q_*\mathcal{E}_{V(\lambda)})_{[1]}=H^0(\mathcal{E}_{V(\lambda)},S/\mu_2)=\bigoplus_{j\equiv\lambda (\operatorname{mod} 2)}V(-\frac{\lambda-j}{2},0,-\frac{\lambda+j}{2})    
$$
Then by our convention of taking dual, we see that the $P$-representation corresponding to the $G$-equivariant bundle $q_*\mathcal{E}_V$  is exactly the infinite direct sum $\bigoplus_{j\equiv\lambda (mod 2)}V({\frac{\lambda-j}{2},0,\frac{\lambda+j}{2}})$, so
$$H^k(\mathcal{E}_{V(\lambda)},G/H)=H^k(q_*\mathcal{E}_{V(\lambda)},G/P)=\bigoplus_{j\equiv\lambda (\operatorname{mod} 2)}H^k(\mathcal{E}_{V(\frac{\lambda-j}{2},0,\frac{\lambda+j}{2})}, G/P) 
$$
In order to let $i_*\mathcal{E}_{V(\lambda)}$ maximal Cohen-Macaulay on $\mathcal{N}_{3,1}$, as we discussed above, this is the same as requiring $H^k(\mathcal{E}_{V(\frac{\lambda-j}{2},0,\frac{\lambda+j}{2})}, G/P)=0$ for $k=1,2$ and all $j\equiv\lambda(\operatorname{mod}2)$. Applying the BWB theorem \ref{BWBparabolic}, we see that this is equivalent to  either $(\frac{\lambda-j+4}{2},1,\frac{\lambda+j}{2})$ being singular or the Weyl group element $w$ permuting it to the dominant weight having length $0$ or $3$. In $SL(3)$ case, this is the same as saying that either $(\frac{\lambda-j+4}{2},1,\frac{\lambda+j}{2})$ has strict increasing or decreasing order, or two of these three integers are the same.

Now let $j=\lambda$, the weight is $(2,1,\lambda)$. In order to satisfy the conditions above, we see that $\lambda\leq 2$. Again let $j=\lambda+4$, the weight becomes $(0,1,\lambda+2)$, and we get $\lambda+2\geq 0$, so $\lambda\geq -2$. Conversely, for any $-2\leq \lambda\leq 2$, one can check one of the conditions above occurs for any $j\equiv \lambda(\operatorname{mod} 2)$. So we have shown:

\begin{proposition}\label{SL3GMCM}
For an irreducible representation $V(\lambda)$ of $H$, the corresponding line bundle $\mathcal{E}_{V(\lambda)}$ on $\mathcal{O}_{\min}$ pushes forward to a maximal Cohen-Macaulay sheaf on $\mathcal{N}_{3,1}$ if and only if $-2\leq \lambda\leq 2$.
\end{proposition}

To compare with our previous result, we should know which Weil divisor $V(\lambda)$ gives. Identify $\mathcal{O}_{\min}$ with $T^*\mathbb{P}^2\setminus s(\mathbb{P}^2)$ where $s$ is the zero section. $SL(3)$ acts on $T^*\mathbb{P}^2\setminus s(\mathbb{P}^2)$ by
$$g. (i,j)=(gi,jg^{-1})
$$
where we identify $T^*\mathbb{P}^2\setminus s(\mathbb{P}^2)$ with pairs $(i,j)$ with $i:\mathbb{C}\rightarrow\mathbb{C}^3$, $j:\mathbb{C}^3\rightarrow\mathbb{C}$, $j\cdot i=0$. Then the natural projection $T^*\mathbb{P}^2\setminus s(\mathbb{P}^2)\rightarrow\mathbb{P}^2$ is $G$-equivariant, it corresponds to the natural quotient $G/H\rightarrow G/Q$ where $Q$ is the parabolic subgroup containing matrices of the following form:
$$\begin{bmatrix}
a & b & c\\
0 & d & e\\
0 & f & g
\end{bmatrix}
$$
On $\mathbb{P}^2$, the character $a$ of $Q$ corresponds to the tautological bundle $G\times_PV_a^\vee=\mathcal{O}_{\mathbb{P}^2}(1)$. Pulling this back to $T^*\mathbb{P}^2\setminus s(\mathbb{P}^2)$, one gets $\mathcal{O}_{T^*\mathbb{P}^2}(1)|_{T^*\mathbb{P}^2\setminus s(\mathbb{P}^2)}$. On the representation side, this gives character $a$ for $H$, so we actually have $i_*\mathcal{E}_{V(1)}=\mathcal{O}(D)$. Then, by the law of taking tensor product and dual, one can easily figure out that $i_*\mathcal{E}_{V(\lambda)}=\mathcal{O}(\lambda D)$. Hence Proposition \ref{SL3GMCM} is the same as saying $\mathcal{O}(mD)$ is maximal Cohen-Macaulay if and only if $-2\leq m\leq 2$. This coincides with our discovery above.

The identification of $i_*\mathcal{E}_{V(\lambda)}$ and $\mathcal{O}(-\lambda D)=\pi_*\mathcal{O}_{T^*\mathbb{P}^2}(-\lambda)$ also helps us compute global sections. From the computation above, we know that
$$H^0(i_*\mathcal{E}_{V(\lambda)},\mathcal{N}_{3,1})=H^0(\mathcal{E}_{V(\lambda)},G/H)=\bigoplus_{j\equiv\lambda (\operatorname{mod} 2)}H^0(\mathcal{E}_{V(\frac{\lambda-j}{2},0,\frac{\lambda+j}{2})}, G/P) 
$$
For a dominant weight $(\lambda_1,\lambda_2,\lambda_3)$, i.e., $\lambda_1>\lambda_2>\lambda_3$, there is a unique corresponding irreducible representation $V_{SL(3)}(\lambda_1,\lambda_2,\lambda_3)$ of $SL(3)$, and by BWB theorem, we have 
$$H^0(\mathcal{E}_{V(\lambda_1,\lambda_2,\lambda_3)}, G/P)=\begin{cases}
V_{SL(3)}(\lambda_1,\lambda_2,\lambda_3)^\vee, & (\lambda_1+2,\lambda_2+1,\lambda_3)\  \text{is dominant}\\
0, & \text{otherwise}
\end{cases}
$$
so we can compute that
$$\bigoplus_{j\equiv\lambda (\operatorname{mod} 2)}H^0(\mathcal{E}_{V(\frac{\lambda-j}{2},0,\frac{\lambda+j}{2})}, G/P) =
\bigoplus_{i\geq\max\{0,\lambda\}}V_{SL(3)}(i,0,\lambda-i)^\vee
$$
On the other hand, 
$$H^0(i_*\mathcal{E}_{V(\lambda)},\mathcal{N}_{3,1})=H^0(\pi_*\mathcal{O}_{T^*\mathbb{P}^2}(-\lambda))=H^0(S^\bullet(T_{\mathbb{P}^2})(-\lambda))
$$
so we conclude that 
$$H^0(S^\bullet(T_{\mathbb{P}^2})(-\lambda))=\bigoplus_{i\geq\max\{0,\lambda\}}V_{SL(3)}(i,0,\lambda-i)^\vee
$$

\subsection{Equivariant sheaves on $\mathfrak{sl}_{n+1}$ minimal orbit closures} The story in the previous section is generalizable to all $A_n$ type minimal orbits.
The action of $SL(n+1)$ on $\mathcal{O}_{\min}=\mathcal{N}_{n+1,1}\setminus 0$ has stabilizer
$$H=\left\{\left.\begin{bmatrix}
a & \alpha^T & c\\
0 & A & \beta\\
0 & 0 & a
\end{bmatrix}\right |\ a\in\mathbb{C}^\times,\ c\in\mathbb{C},\ \alpha, \beta \in \mathbb{C}^{n-1},\ A\in GL(n-1),\ a^2\det A=1\right\}
$$
The irreducible $H$-representations are exactly the irreducible representations of the $G/U(G)$ where $U(G)$ is the unipotent radical consisting of matrices of the form
$$
\begin{bmatrix}
1 & \alpha^T & c\\

0 & \operatorname{Id} & \beta\\
0 & 0 & 1
\end{bmatrix}
$$
Hence the irreducible representations are essentially representations of the group of quasi-diagonal matrices $D=G/U(G)=\{\operatorname{diag}(a, A, a)|\ a\in \mathbb{C}^\times,\ A\in GL(n-1),\ \ a^2\det A=1\}$. 

As before, we consider $S=\{\operatorname{diag}(x, \operatorname{Id},x^{-1})|x\in\mathbb{C}^\times\}$. One sees that $S\cap H=\{\operatorname{diag}(\pm 1, \operatorname{Id}, \pm 1)\}\cong\mu_2$, $S$ normalize $H$ and the product $P=S\cdot H\subset SL(n+1)$ is the parabolic group consisting of matrices of the following form
$$\begin{bmatrix}
    \ast & \ast & \cdots & \ast & \ast\\
    0 & \ast & \cdots & \ast & \ast\\
    \vdots & \vdots  & \vdots & \vdots &\vdots\\
    0 & \ast & \cdots & \ast & \ast\\
    0 & 0 & \cdots & 0 & \ast
\end{bmatrix}
$$
One also sees that elements in $S$ commute with elements in $D$. Again we have an affine morphism $q:G/H\rightarrow G/P$ with fiber  $P/H$ at the identity class $[1]$. The injection $S\hookrightarrow P$ induces an isomorphism $S/\mu_2\rightarrow P/H$. For any irreducible $H$-representation $V$, we have 
$$(q_*\mathcal{E}_V)_{[1]}=H^0(\mathcal{E}_V,P/H)=H^0(\mathcal{E}_V,S/\mu_2)=\begin{cases}
    \displaystyle \bigoplus_{i\in\mathbb{Z}} \mathbb{C}t^{2i}\otimes V^\vee, & \rho|_{\mu_2}\  \text{is trivial}\\
    \displaystyle \bigoplus_{i\in\mathbb{Z}} \mathbb{C}t^{2i+1}\otimes V^\vee, & \rho|_{\mu_2} \ \text{is nontrivial}
\end{cases}
$$
Since $V$ is irreducible, the unipotent of $H$ acts trivially. So if we suppose $\rho:H\rightarrow GL(V)$ is the representation, by the commutativity of $D$ and $S$, we have $\rho(h)=\rho(shs^{-1})$ for any $s\in S,\ h\in H$. Then one sees that $\mathbb{C}t^j\otimes V^\vee$ is $P$-stable and the $P$-action is given by
$yh.(t^j\otimes \phi)=t^j\otimes y^{-j}\rho^\vee(h) \phi$. One also check easily that these are all irreducible $P$ representations.

Now let's discuss the irreducible representations of $D$. It's easy to see that
$$D\cong\begin{cases}
SL(n-1)\times\mathbb{G}_m/\mu_{\frac{n-1}{2}} & n\  \text{is odd}\\
SL(n-1)\times\mathbb{G}_m/\mu_{n-1} & n\  \text{is even}
\end{cases} 
$$
where we identify $SL(n-1)$ with $\operatorname{diag}(1,SL(n-1),1)$ and $\mathbb{G}_m$ with $\{\operatorname{diag}(a^{\frac{n-1}{2}},a^{-1}\operatorname{Id},a^{\frac{n-1}{2}})|a\in\mathbb{C}^\times\}$ for odd $n$, $\{\operatorname{diag}(a^{n-1},a^{-2}\operatorname{Id},a^{n-1})|a\in\mathbb{C}^\times\}$ for even $n$. So the irreducible representations of $D$ are essentially irreducible representations of $SL(n-1)$ with an action of $\mathbb{G}_m$ that is compatible with the $SL(n-1)$ action.
Then one sees that the irreducible representations are determined by the dominant weights in the weight lattice of $D$: $\mathbb{Z}^n/\mathbb{Z}(2,1,1,...,1)$, where we say a weight $(\lambda,\lambda_1,\lambda_2,...,\lambda_{n-1})$ is dominant if $\lambda_1\geq\lambda_2\geq...\geq\lambda_{n-1}$.

Suppose $V$ is an irreducible representation with highest weight $(\lambda,\lambda_1,\lambda_2,...,\lambda_{n-1})$ and corresponding eigenvector $v$. Denote by $v^*$ the lowest vector in $V^\vee$. Then the action of the diagonal torus on $\mathbb{C}t^{-j}\otimes v^*$ is by
\begin{align}\notag
\operatorname{diag}(ax,a_1,a_2,...,a_{n-1},\frac{a}{x}).(t^{-j}\otimes v^*)&=a^{-\lambda} a_1^{-\lambda_1} a_2^{-\lambda_2}...a_{n-1}^{-\lambda_{n-1}}x^{j}(t^j\otimes v)\\\notag
&=(ax)^{-\frac{\lambda-j}{2}}a_1^{-\lambda_1} a_2^{-\lambda_2}...a_{n-1}^{-\lambda_{n-1}}(\frac{a}{x})^{-\frac{\lambda+j}{2}}(t^j\otimes v)
\end{align}
So $\mathbb{C}t^{-j}\otimes V^\vee$ is the dual of the irreducible $P$-representation with highest weight $(\frac{\lambda-j}{2},\lambda_1,\lambda_2,...,\lambda_{n-1},\frac{\lambda+j}{2})$. Therefore, $q_*\mathcal{E}_V=\bigoplus_{j\equiv\lambda (\operatorname{mod} 2)}\mathcal{E}(\frac{\lambda-j}{2},\lambda_1,\lambda_2,...,\lambda_{n-1},\frac{\lambda+j}{2})$ on $G/P$. Since $G/H\rightarrow G/P$ is affine: 
$$H^k(\mathcal{E}_V,G/H)=H^k(q_*\mathcal{E}_V,G/P)=\bigoplus_{j\equiv\lambda (\operatorname{mod} 2)}H^k(\mathcal{E}(\frac{\lambda-j}{2},\lambda_1,\lambda_2,...,\lambda_{n-1},\frac{\lambda+j}{2}),G/P)
$$

Now let $i:G/H\rightarrow \mathcal{N}_{n+1,1}$ be the open embedding. By Proposition \ref{localseqMCM}, $i_*\mathcal{E}_V$ is maximal Cohen Macaulay if and only if $H^k(\mathcal{E}_V,G/H)=0$ for $1\leq k\leq 2n-2$. Applying the BWB theorem \ref{BWBparabolic} for $(SL(n+1),P)$, since we already have $\lambda_1\geq\lambda_2\geq...\geq\lambda_{n-1}$, we see that $(\frac{\lambda-j+2n}{2},\lambda_1+n-1,\lambda_2+n-2,...,\lambda_{n-1}+1,\frac{\lambda+j}{2})$ satisfies one of the following conditions:
\begin{enumerate}
    \item There are two coordinates having the same value in this weight
    \item $\frac{\lambda-j+2n}{2}>\lambda_1+n-1>\lambda_2+n-2>...>\lambda_{n-1}+1>\frac{\lambda+j}{2}$
    \item $\frac{\lambda+j}{2}>\lambda_1+n-1>\lambda_2+n-2>...>\lambda_{n-1}+1>\frac{\lambda-j+2n}{2}$
\end{enumerate}
Let $k=\frac{\lambda-j+2n}{2}$, then $\frac{\lambda+j}{2}=\lambda+n-k$. When $j$ goes over all integers $\equiv\lambda(\operatorname{mod}2)$, $k$ goes over all integers. One should notice that $k=\lambda+n-k$ if and only if $\lambda\equiv n (\operatorname{mod}2)$ and $k=\frac{\lambda+n}{2}$. Now define the following set consisting of "gaps":
$$
A=\{N|N\in\mathbb{Z}, \lambda_{n-1}+1<N<\lambda_{1}+n-1, N\neq\lambda_{i}+n-i\ \text{for any}\ 1\leq i\leq n-1, N\neq\frac{\lambda+n}{2}\} 
$$
We also suppose $M=\max\{i|1\leq i\leq n-1, \lambda_i=\lambda_1\}$ and $m=\min\{i|1\leq i\leq n-1, \lambda_i=\lambda_{n-1}\}$. In order to meet the three conditions above for all $j$, we should have, 
\begin{enumerate} 
\item When $k>\lambda_1+n-1$, we should have $\lambda+n-k\leq\lambda_m+n-m$ or $k=\frac{\lambda+n}{2}$. 
\item When $k=\alpha$ for some $\alpha\in A$, we should have $\lambda+n-k=\lambda_{i}+n-i$ for some $i$. 
\item When $k<\lambda_{n-1}+1$, there should be $\lambda+n-k\geq\lambda_M+n-M$ or $k=\frac{\lambda+n}{2}$. 
\end{enumerate} 
From (2) we know that for each $t$, $\lambda+n-\alpha\in\{\lambda_i+n-i|1\leq i \leq n-1\}$, so $$\lambda\in\displaystyle\bigcap_{\alpha\in A}\{\lambda_i+\alpha-i|1\leq i \leq n-1\}$$ 
From (1) and (3), we derive that when $\lambda_1>\lambda_{n-1}$, we should have $\lambda\leq\lambda_1+\lambda_m+n-m$ and $\lambda\geq\lambda_{n-1}+\lambda_M-M$, since $A\neq\varnothing$, and because of that, the condition $k=\lambda+n-k$ cannot hold in cases $k>\lambda+n-1$ or $k<\lambda_{n-1}+1$. When $\lambda_1=\lambda_2=...=\lambda_{n-1}=\mu$, we could have $\lambda=\lambda_1+\lambda_m+n-m+1=2\mu+n$ or $\lambda=\lambda_{n-1}+\lambda_M-M-1=2\mu-n$. This is because, for example, if $\lambda=2\mu+n$, when $k=\lambda_1+n=\mu+n$, we have $k=\frac{2\mu+2n}{2}=\frac{\lambda+n}{2}$; while when $k>\lambda_1+n$, we still have $\lambda+n-k=\lambda_{i}+n-i$. And one can also check that (3) holds as well. In this case, one can also observe that $A=\varnothing$, so (2) is actually free.

These conditions above are all we need for maximal Cohen-Macaulay. Hence, we have proved: 
\begin{theorem}\label{SLMCM} Suppose $\mathcal{O}_{\min}$ is the minimal orbit in $\mathfrak{sl}_{n+1}$ $(n\geq 2)$ with respect to the conjugation action of $SL(n+1)$ and $H$ is the stabilizer. Given an equivariant bundle $\mathcal{E}_V=SL(n+1)\times_H V^\vee$ corresponding to an irreducible representation $V$ of $H$ determined by a weight $(\lambda,\lambda_1,\lambda_2...\lambda_{n-1})$ in $\mathbb{Z}^n/\mathbb{Z}(0,1,...,1)$ with $\lambda_1\geq\lambda_2\geq...\geq\lambda_{n-1}$, the pushforward $i_*\mathcal{E}_V$ is maximal Cohen-Macaulay on orbit closure $\overline{\mathcal{O}_{\min}}=\mathcal{N}_{n+1,1}$ if and only if 
\begin{enumerate} 
\item If $\lambda_1=\lambda_2=\cdots=\lambda_{n-1}=\mu$, then $2\mu-n\leq\lambda\leq2\mu+n$.
\item If $\lambda_1>\lambda_{n-1}$, there is $\lambda_{n-1}+\lambda_M-M\leq\lambda\leq\lambda_1+\lambda_m+n-m$ and $$\lambda\in\displaystyle\bigcap_{\alpha\in A}\{\lambda_i+\alpha-i|1\leq i \leq n-1\}
$$
\end{enumerate}
where $M=\max\{i|1\leq i\leq n-1, \lambda_i=\lambda_1\}$, $m=\min\{i|1\leq i\leq n-1, \lambda_i=\lambda_{n-1}\}$ and $$A=\{N|N\in\mathbb{Z}, \lambda_{n-1}+1<N<\lambda_{1}+n-1, N\neq\lambda_{i}+n-i\ \text{for any}\ 1\leq i\leq n-1, N\neq\frac{\lambda+n}{2}\} $$ 
\end{theorem}

Let's again apply this to the line bundles. The $1$-dimensional representations of $H$ are exactly classified by $(\lambda,\lambda_1,\lambda_2...\lambda_{n-1})$ where $\lambda_1=\lambda_2=...=\lambda_{n-1}=\mu$ for some $\mu\in\mathbb{Z}$. So condition (1) tells you that $-n\leq\lambda-2\mu\leq n$. As argued in the $SL(3)$ case, one see that this coincides with our observation in \ref{A_nMCMdivisors}.

\section{Equivariant Maximal Cohen-Macaulay sheaves on $C_n$ type minimal orbit closures}
In this chapter, we generalize the story of maximal Cohen-Macaulay sheaves on minimal orbit from the previous chapter to the symplectic groups $Sp(2n)$ ($n\geq 2$).

Take a standard basis $e_1,e_2,...,e_n,f_1,f_2,...,f_n$ in a symplectic space, the standard symplectic form has the matrix $J=\begin{bmatrix}
    0 & \operatorname{Id}_n\\
    -\operatorname{Id}_{n} & 0
\end{bmatrix}$, and one typical element in the minimal orbit of $\mathfrak{sp}_{2n}$ is $E_{1,n+1}$ sending $f_1$ to $e_1$ and other basis elements to $0$. However, the stabilizer of this element is not block diagonal. For a smooth generalization of our previous results, we swap $f_1$ and $f_n$. Then the symplectic quadratic form becomes $J'=
\begin{bmatrix}0 & I'\\-I'& 0\end{bmatrix}$ where $I'$ is the permutation matrix $\begin{bmatrix}
    0 & 0 & 1\\
    0 & \operatorname{Id}_{n-1} & 0\\
    1 & 0 & 0
\end{bmatrix}$ and the typical element in the minimal orbit becomes $E_{1,2n}$. For $g=(a_{i,j})\in Sp(2n)$, the condition $g E_{1,2n}=E_{1,2n}g$ is $a_{1,1}=a_{2n,2n}$, $a_{i,1}=a_{2n,j}=0$ for $1\leq i,j\leq 2n$. So all $g\in \operatorname{Stab_{Sp(2n)}}(E_{1,2n})$ can be written as
$$g=\begin{bmatrix}
    a & \alpha  & b\\
    0 & A & \beta\\
    0 & 0 & a
\end{bmatrix},\ a\in\mathbb{C}^\times,\ b\in\mathbb{C},\ \alpha,\beta^T\in\mathbb{C}^{1\times (2n-2)}, A\in GL(2n-2)
$$
Now we need to figure out the necessary relations for $g$ to be in $Sp(2n)$. Write $J'=\begin{bmatrix}
    0 & 0  & 1\\
    0 & J'' & 0\\
    -1 & 0 & 0
\end{bmatrix}$ where 
$$J''=\begin{bmatrix}
    0 & 0 & 0 & \operatorname{Id}_{n-2}\\
    0 & 0 & 1 & 0\\
    0 & -1 & 0 & 0\\
    -\operatorname{Id}_{n-2} & 0 & 0 & 0
\end{bmatrix}
$$
Then $g^TJ'g=J'$ translates into 
$$\begin{bmatrix}
    0 & 0  & a^2\\
    0 & A^TJ''A & A^TJ''\beta+a\alpha^T\\
    -a^2 & -a\alpha+\beta J'' A & \beta^T J''\beta
\end{bmatrix}=\begin{bmatrix}
    0 & 0  & 1\\
    0 & J'' & 0\\
    -1 & 0 & 0
\end{bmatrix}
$$
Since $J''$ is skew-symmetric, we have $A^TJ''\beta+a\alpha^T=-(-a\alpha+\beta J'' A)^T$ and $\beta^T J''\beta=0$. Hence the matrix equation is equivalent to
$$A^TJ''A=J'',\ a^2=1, \ \alpha=\frac{1}{a}\beta^TJ'' A
$$
As in the previous chapter, denote $H= \operatorname{Stab_{Sp(2n)}}(E_{1,2n})$, we see from the relations above that the unipotent radical of $H$ is
$$U(H)=\left\{\left.\begin{bmatrix}
    1 & \beta^TJ''  & b\\
    0 & \operatorname{Id}_{2n-2} & \beta\\
    0 & 0 & 1
\end{bmatrix}\right|\ b\in\mathbb{C},\ \beta\in\mathbb{C}^{(2n-2)\times 1}\right\}
$$
Since $J''$ also serves as a symplectic form on some $2n-2$-dimensional space, we see that $H/U(H)\cong Sp(2n-2)\times\mu_2$, where the components are identified with $Sp(2n-2)=Sp(J'')=\{\operatorname{diag}(1,A,1)|A^TJ''A=J''\}$ and $\mu_2=\{\operatorname{diag}(\pm 1,\operatorname{Id}_{2n-2},\pm 1)\}$.

Consider $S=\{\operatorname{diag}(a,\operatorname{Id}_{2n-2},a^{-1})|a\in\mathbb{C}^\times\}$. As in the previous chapter, one sees that $S\cong\mathbb{G}_m$, $S\cap H\cong \mu_2$, $S$ normalizes $H$ and that $P=S\cdot H$ is parabolic. Indeed, with the induced action of $Sp_{2n}$ on the projectivization of the minimal nilpotent orbit $\mathcal{O}_{\min}(\mathfrak{sp}_{2n})$, $P$ is the stabilizer of the point representing the line $\mathbb{C}e_1$. For an irreducible representation $\rho:H\rightarrow GL(V)$ of $H$, we consider the corresponding vector bundle $\mathcal{E}_V$ on $G/H$. Push it forward along the quotient map $q:G/H\rightarrow G/P$, as before we get 
$$(q_*\mathcal{E}_V)_{[1]}=\begin{cases}
    \displaystyle \bigoplus_{i\in\mathbb{Z}} \mathbb{C}t^{2i}\otimes V^\vee, & \rho|_{\mu_2}\  \text{is trivial}\\
    \displaystyle \bigoplus_{i\in\mathbb{Z}} \mathbb{C}t^{2i+1}\otimes V^\vee, & \rho|_{\mu_2} \ \text{is nontrivial}
\end{cases}
$$
as $P$ representations. Now, an irreducible representation $V$ of $H$ is the same as an irreducible representation of $Sp_{J''}(2n-2)$ and with $\mu_2$ acting by $\pm 1$. Consider the following maximal torus in $Sp_{J''}(2n-2)$
$$T=\{\operatorname{diag}(t_1,t_2,...,t_{n-2},t_{n-1},t_{n-1}^{-1},t_{1}^{-1},t_{2}^{-1},...,t_{n-2}^{-1})|t_{i}\in\mathbb{C}^{\times}, 1\leq i \leq n-1\}
$$
With this maximal torus, the dominant integral weights are the characters $(\lambda_1,\lambda_2,...,\lambda_{n-1})$ of $T$ such that $\lambda_1\geq\lambda_2\geq...\geq \lambda_{n-1}\geq 0$. Then the dominant weights of $H$ can be encoded by $(\lambda_0,\lambda_1,...,\lambda_{n-1})$ in $\mathbb{Z}/2\times\mathbb{Z}^{n-1}$ with $\lambda_1\geq\lambda_2\geq...\geq\lambda_{n-1}\geq 0$. The action of $P$ on $\mathbb{C}t^{-j}\otimes V^\vee$ has lowest weight $\lambda=(j,-\lambda_1,-\lambda_2,...,-\lambda_{n-1})$. So it is the dual representation of the representation with highest weight $(-j,\lambda_1,\lambda_2,...,\lambda_{n-1})$. Hence
$$q_*\mathcal{E}_V=\bigoplus_{j\equiv\lambda_0(\operatorname{mod}2)}\mathcal{E}(-j,\lambda_1,\lambda_2,...,\lambda_{n-1})
$$
Since $q$ is affine, we have
$$H^m(\mathcal{E}_V,G/H)=H^m(q_*\mathcal{E}_V,G/P)=\bigoplus_{j\equiv\lambda_0 (\operatorname{mod} 2)}H^m(\mathcal{E}(-j,\lambda_1,\lambda_2,...,\lambda_{n-1}),G/P)
$$
The half sum of positive roots for $Sp(2n)$ is $(n,n-1,...,1)$. Then $\lambda+\rho=(n-j,\lambda_1+n-1,\lambda_2+n-2,...,\lambda_{n-1}+1)$. As we have supposed $\lambda_1\geq\lambda_2\geq...\geq\lambda_{n-1}\geq 0$, we have $\lambda_1+n-1>\lambda_2+n-2>...>\lambda_{n-1}+1>0$. So we have $\lambda+j$ is singular if and only if $n-j=0$ or $\pm(\lambda_i+n-i)$ for some $1\leq i\leq n-1$. When $|n-j|\notin\{0,\lambda_1+n-1,\lambda_2+n-2,...,\lambda_{n-1}+1\}$, suppose $w$ is the Weyl group element sending $\lambda+\rho$ to a dominant weight, then 
$$l(w)=\begin{cases}
    \#\{i|\lambda_i+n-i>n-j\} & n-j>0\\
    2n-1-\#\{i|\lambda_i+n-i>j-n\} & n-j<0
\end{cases}
$$
By Proposition \ref{localseqMCM}, the pushforward of $\mathcal{E}_V$ on minimal orbit closure is maximal Cohen-Macaulay if and only if $H^m(\mathcal{E}_V,G/H)=\bigoplus_{j\equiv\lambda_0 (\operatorname{mod} 2)}H^m(\mathcal{E}(-j,\lambda_1,\lambda_2,...,\lambda_{n-1}),G/P)=0$ for $1\leq m \leq 2n-2$. By BWB theorem \ref{BWBparabolic}, this is the same as requiring that $\lambda+\rho$ is either singular, or $l(w)\in\{0,2n-1\}$. Again, define the set of gaps
$$A=\{N|N\in\mathbb{Z}, \ |N|<\lambda_{1}+n-1,\  N\neq 0,\   \ N\neq \pm(\lambda_{i}+n-i)\ \  \text{for any}\ \   1\leq i\leq n-1\}
$$
We see that it's equivalent to say that elements in $A$ cannot be congruent to $n-\lambda_0$ modulo $2$. So we have proved
\begin{theorem}\label{Ctypetheorem}
Suppose $\mathcal{O}_{\min}$ is the minimal orbit in $\mathfrak{sp}_{2n}$ $(n\geq 2)$ with respect to the conjugation action of $Sp(2n)$ and $H$ is the stabilizer. An equivariant bundle $\mathcal{E}_V=Sp(2n)\times_H V^\vee$ corresponding to an irreducible representation $V$ of $H$ is determined by the weight $(\lambda_0,\lambda_1,\lambda_2...\lambda_{n-1})$ in $H$'s lattice $\mathbb{Z}/2\times\mathbb{Z}^{n-1}$. The pushforward $i_*\mathcal{E}_V$ is maximal Cohen-Macaulay on the orbit closure $\overline{\mathcal{O}_{\min}}$ if and only if 
for any element $k$ in the following set
$$A=\{N|N\in\mathbb{Z}, \ |N|<\lambda_{1}+n-1,\  N\neq 0,\   \ N\neq \pm(\lambda_{i}+n-i)\ \  \text{for any}\ \   1\leq i\leq n-1\}
$$
we have $k\not\equiv n-\lambda_0 (\operatorname{mod} 2)$.
\end{theorem}

\section{Equivariant maximal Cohen-Macaulay sheaves on $B_n$, $D_n$ type minimal orbit closures}
In this chapter, we generalize our story to $\mathfrak{so}_{2n}$ and $\mathfrak{so}_{2n+1}$. Let's suppose $n\geq 3$ to avoid the cases we have discussed. 

The corresponding Lie groups are $\operatorname{Spin}(2n)$ and $\operatorname{Spin}(2n+1)$, which are the central extensions of $SO(2n)$ and $SO(2n+1)$ with $\mu_2$ respectively. However, since the center $\mu_2$ acts trivially on the lie algebra $\mathfrak{so}_{2n}$ and $\mathfrak{so}_{2n+1}$, the $G=\operatorname{Spin}(2n)$ or $\operatorname{Spin}(2n+1)$ equivariant sheaves are exactly $G=SO(2n)$ or $SO(2n+1)$ equivariant sheaves with an action of $\mu_2$ on fibers. If we are mainly interested in whether those sheaves are maximal Cohen-Macaulay, it is the same to consider those two equivariant structures. We will use the matrix groups to show the steps, for their advantages of being visible, and state results for  $\operatorname{Spin}$ in the next chapter where we develop method for abstract simple Lie algebras.

\subsection{On the minimal orbit closure of $\mathfrak{so}_{2n}$ }Let's first work on $SO(2n)$. Take a standard basis $e_1,e_2,...,e_n,f_1,f_2,...,f_n$, the standard symmetric form has the matrix $J=\begin{bmatrix}
    0 & \operatorname{Id}_n\\
    \operatorname{Id}_{n} & 0
\end{bmatrix}$. One typical element in the minimal orbit of $\mathfrak{so}_{2n}$ is $E_{1,n+2}-E_{2,n+1}$. However, the stabilizer of this element is not block diagonal. For a smooth generalization of our previous results, we change the order of the basis to $e_1,e_2,...,e_n,f_3,...,f_n, f_1,f_2$. Then the symmetric quadratic form becomes $J'=
\begin{bmatrix}0 & I'\\I'^T& 0\end{bmatrix}$ where $I'$ is the permutation matrix $\begin{bmatrix}
    0 & \operatorname{Id}_2\\
    \operatorname{Id}_{n-2} & 0 
\end{bmatrix}$ and the typical element in the minimal orbit becomes $E_{1,2n}-E_{2,2n-1}$. For $g=(a_{i,j})\in SO(2n)$, the condition $g(E_{1,2n}-E_{2,2n-1})=(E_{1,2n}-E_{2,2n-1})g$ can be computed as 
\begin{enumerate}
    \item $a_{i,1}=a_{i,2}=0$ for $3\leq i\leq 2n$, $a_{2n,j}=a_{2n-1,j}=0$ for $1\leq j \leq 2n-2$
    \item $a_{1,1}=a_{2n,2n}$, $a_{2,2}=a_{2n-1,2n-1}$, $a_{12}=-a_{2n,2n-1}$ and $a_{21}=-a_{2n-1,2n}$
\end{enumerate}
So the matrices in the stabilizer are in the following form
$$\begin{bmatrix}
U & P & B\\
0 & A & Q\\
0 & 0 & (\det U)(U^{-1})^T
\end{bmatrix}
$$
where $U\in GL(2)$, $A\in GL(2n-4)$, $P,Q^T\in \mathbb{C}^{2\times (2n-4)}$, $B\in\mathbb{C}^{2\times 2}$. Write $J'=\begin{bmatrix}
    0 & 0 &\operatorname{Id}_2\\
    0 & J'' & 0\\
    \operatorname{Id}_2 & 0 & 0
\end{bmatrix}$ where $J''=\begin{bmatrix}
    0 & \operatorname{Id}_{n-2}\\
    \operatorname{Id}_{n-2} & 0
\end{bmatrix} $. Then $g^{T}J'g=J'$ translates into 
$$\begin{bmatrix}
    0 & 0 & (\det U)\operatorname{Id}_2\\
    0 & A^TJ''A & A^TJ''Q+(\det U) P^T(U^{-1})^T\\
    (\det U)\operatorname{Id}_2 & (\det U)U^{-1} P+Q^TJ''A & (\det U )U^{-1}B+Q^TJ''Q+(\det U)B^T(U^{-1})^T
\end{bmatrix}=\begin{bmatrix}
    0 & 0 &\operatorname{Id}_2\\
    0 & J'' & 0\\
    \operatorname{Id}_2 & 0 & 0
\end{bmatrix}
$$
We see that $\det U=1$, hence the condition is equivalent to 
$$U\in SL(2), \ A\in SO(J''),\ U^{-1}P+Q^TJ''A=0,\ U^{-1}B+(U^{-1}B)^T+Q^TJ''Q=0
$$
Denote by $H=\operatorname{Stab}_{SO(2n)}(E_{1,2n}-E_{2,2n-1})$. One concludes, as in the previous chapters, the unipotent part $U(H)$ is the group with identities on diagonal blocks and $H/U(H)$ is the group $SL(2)\times SO(J'')\cong SL(2)\times SO(2n-4)$.

Now consider the following copy of $\mathbb{C}^\times$ in $SO(2n)$: $S=\{\operatorname{diag}(x\operatorname{Id}_2, \operatorname{Id}_{2n-4}, x^{-1}\operatorname{Id}_2)|x\in\mathbb{C}^{\times}\}$. One checks easily that $S\cap H=\{\operatorname{diag}(\pm\operatorname{Id}_2, \operatorname{Id}_{2n-4}, \pm\operatorname{Id}_2)|x\in\mathbb{C}^{\times}\}\cong \mu_2$, $S$ normalize $H$ and $P=SH$ is parabolic. Indeed, with the induced action of $P$ on the projectivization of the minimal nilpotent orbit $\mathcal{O}_{\min}(\mathfrak{so}_{2n})$, $P$ is the stabilizer of the line $\mathbb{C}(E_{1,2n}-E_{2,2n-1})$.

Write $S=\operatorname{Spec}\mathbb{C}[t,t^{-1}]$. For an irreducible representation $\rho:H\rightarrow GL(V)$ of $H$, we consider the corresponding vector bundle $\mathcal{E}_V$ on $G/H$. Push it forward along the quotient map $q:G/H\rightarrow G/P$, as before we get 
$$(q_*\mathcal{E}_V)_{[1]}=\begin{cases}
    \displaystyle \bigoplus_{i\in\mathbb{Z}} \mathbb{C}t^{2i}\otimes V^\vee, & \rho|_{\mu_2}\  \text{is trivial}\\
    \displaystyle \bigoplus_{i\in\mathbb{Z}} \mathbb{C}t^{2i+1}\otimes V^\vee, & \rho|_{\mu_2} \ \text{is nontrivial}
\end{cases}
$$
as $P$ representations, and each summand $\mathbb{C}t^j\otimes V$ is irreducible. Now any irreducible representation of $H$ is an irreducible representation of $H/U(H)\cong SL(2)\times SO(2n-4)$, which is can be classified by $W\boxtimes W'$ where $W$ is an irreducible $SL(2)$ module and $W'$ is an irreducible $SO(2n-4)$ module.

Irreducible $SO(2n-4)$ representations are classified by dominant integral weights $(\lambda_1,\lambda_2,...,\lambda_{n-2})$ where $\lambda_i\in\mathbb{Z}$, $\lambda_1\geq\lambda_2\geq...\geq\lambda_{n-3}\geq |\lambda_{n-2}|$. While the $SL(2)$ irreducible representations classified by nonnegative integers, all representations of $V$ can be classified by weight in the lattice $\mathbb{Z}^{n-1}$. For a weight $(\lambda, \lambda_1,\lambda_2,...,\lambda_{n-2})$ where $\lambda\geq 0,\  \lambda_1\geq\lambda_2\geq...\geq\lambda_{n-3}\geq |\lambda_{n-2}|$.

Now consider the following maximal torus in $P$
$$T=\{\operatorname{diag}(t_1,...,t_n, t_3^{-1},...,t_n^{-1},t_1^{-1},t_2^{-1})| t_1,...,t_n\in\mathbb{C}^{\times}\}
$$
Replace $t_1$ with $ax$ and $t_2$ with $a^{-1}x$, the action of $T$ on the lowest weight vector $\mathbb{C}t^{-j}\otimes v^*$ is given by 
$$\operatorname{diag}(ax,a^{-1}x,t_3,....,t_n,t_3^{-1},...,t_n^{-1},a^{-1}x^{-1},ax^{-1}) (t^{-j}\otimes v^*)=x^{j}a^{-\lambda}t_3^{-\lambda_1}t_4^{-\lambda_2}...t_n^{-\lambda_{n-2}}(t^{-j}\otimes v^*)
$$
While $x^{j}a^{-\lambda}=(ax)^{-\frac{\lambda-j}{2}}(a^{-1}x)^{-\frac{-\lambda-j}{2}}$, we see that $\mathbb{C}t^{-j}\otimes V^\vee$ is the dual representation of the irreducible representation with the highest weight $(\frac{\lambda-j}{2},\frac{-\lambda-j}
{2},\lambda_1,...,\lambda_{n-2})$. So 
$$q_*\mathcal{E}_V=\bigoplus_{j\equiv\lambda(\operatorname{mod}2)}\mathcal{E}(\frac{\lambda-j}{2},\frac{-\lambda-j}
{2},\lambda_1,...,\lambda_{n-2})
$$
Now the dimension of minimal orbit of $\mathfrak{so}_{2n}$ is $4n-6$. By Proposition \ref{localseqMCM}, the pushforward of $\mathcal{E}_V$ on minimal orbit closure is maximal Cohen-Macaulay if and only if $$H^m(\mathcal{E}_V,G/H)=\bigoplus_{j\equiv\lambda (\operatorname{mod} 2)}H^m(\mathcal{E}(\frac{\lambda-j}{2},\frac{-\lambda-j}
{2},\lambda_1,...,\lambda_{n-2}),G/P)=0
$$ 
for $1\leq m \leq 4n-8$. By BWB theorem \ref{BWBparabolic}, this is equivalent to that $(\frac{\lambda-j}{2},\frac{-\lambda-j}
{2},\lambda_1,...,\lambda_{n-2})+\rho$ is either singular, or the Weyl group element $w$ sending it to a dominant weight has $l(w)\in\{0,4n-7\}$. Here $\rho$ for $\mathfrak{so}_{2n}$ is $(n-1,n-2,...,0)$. Let $k=\frac{\lambda-j}{2}$, then $\frac{-\lambda-j}{2}=k-\lambda$, when $j$ goes over all integers $\equiv \lambda (\operatorname{mod} 2)$, $k$ goes through all integers, so the weight we are considering is $\Lambda=(k+n-1,k-\lambda+n-2,\lambda_1+n-3,...,\lambda_{n-2})$. 

Since $\lambda\geq 0$, we have $k+n-1> k-\lambda+n-2$. Also we see by $\lambda_1\geq\lambda_2\geq...\geq\lambda_{n-3}\geq |\lambda_{n-2}|$ that $\lambda_1+n-3>...>\lambda_{n-3}+1>|\lambda_{n-2}|$. So when $\Lambda$ is nonsingular, the length of $w$ can be computed by:
\begin{align}\notag
l(w)&=\{\alpha\in\Phi^+|\langle\Lambda,\alpha^\vee\rangle<0\}\\\notag
&=\#\{i|k+n-1<\lambda_i+n-2-i\}+\#\{i|k-\lambda+n-2<\lambda_i+n-2-i\}\\\notag
&\quad+\ \#\{i|k+n-1<-\lambda_i-n+2+i\}+\#\{i|k-\lambda+n-2<-\lambda_i-n+2+i\}\\\notag
&\quad +\  1_{k+n-1+k-\lambda+n-2<0}\\\notag
&\leq 4\times(n-2)+1\\\notag
&=4n-7
\end{align}
The equal sign holds if and only if all four sets have cardinality $n-2$ and $2k+2n-\lambda-3<0$, i.e., $k< -n+\frac{\lambda+3}{2}$. Since we have already supposed $\lambda\geq 0$ and $\lambda_1>\lambda_2\geq...\geq\lambda_{n-3}\geq |\lambda_{n-2}|$, we see that the first condition above is just $k+n-1<-(\lambda_1+n-3)$. When this holds, since $n\geq 2$, we see that 
$$k<-\lambda_1-2n+4=-n+4-\lambda_1-n\leq -n+2
$$
then we have $k\leq-n+1<-n+\frac{\lambda+3}{2}$, so the second condition automatically holds.

On the other hand, we also have $l(w)\geq 0$ with equal sign holds if and only if $k-\lambda+n-2>\lambda_1+n-3$, i.e., $k>\lambda+\lambda_1-1$. In summary, we have that when $\Lambda$ is nonsingular, there should be
\begin{enumerate}
    \item $l(w)=4n-7 \iff k<-\lambda_1-2n+4$.
    \item $l(w)=0 \iff k>\lambda+\lambda_1-1 $.
\end{enumerate}
So for $-\lambda_1-2n+4\leq k \leq \lambda+\lambda_1-1$, $\Lambda$ should be singular, then there is $k+n-1=\pm(\lambda_i+n-2-i)$ or $k-\lambda+n-2=\pm(\lambda_i+n-2-i)$ or $k+n-1+k-\lambda+n-2=0$. Again define the set of gaps
$$A=\{N|N\in\mathbb{Z},\  -\lambda_1-2n+4\leq N\leq \lambda+\lambda_1-1, \ N\neq\lambda_i-i-1,\ N\neq-\lambda_i-2n+3+i\ \text{for all}\ i\}
$$
We see that when $k\in A$, then $k+n-1\neq\pm(\lambda_i+n-2-i)$, so $k-\lambda+n-2=\pm(\lambda_i+n-2-i)$ or $2k+2n-3-\lambda=0$. Then $\lambda\in\{k-\lambda_i+i,k+2n-4+\lambda_i-i,2k+2n-3\}$. So $q_*\mathcal{E}_V$ is maximal Cohen-Macaulay if and only if 
$\lambda$ is in the set $\bigcap_{k\in A}(\bigcup_{i=1}^{n=2}\{k-\lambda_i+i,k+2n-4+\lambda_i-i\}\cup\{2k+2n-3\})$. So in summary,

\begin{theorem}\label{Dtypetheorem}
Suppose $\mathcal{O}_{\min}$ is the minimal orbit in $\mathfrak{so}_{2n}$ $(n\geq 3)$ with respect to the conjugation action of $SO(2n)$ , and $H$ is the stabilizer. An equivariant bundle $\mathcal{E}_V=SO(2n)\times_H V^\vee$ corresponding to an irreducible representation $V$ of $H$ is determined by a dominant weight $(\lambda,\lambda_1,\lambda_2...\lambda_{n-2})$ with $\lambda\geq 0$ and $\lambda_1\geq...\geq\lambda_{n-3}\geq|\lambda_{n-2}|$ in $H$'s lattice $\mathbb{Z}^{n}$. The pushforward $i_*\mathcal{E}_V$ is maximal Cohen-Macaulay on the orbit closure $\overline{\mathcal{O}_{\min}}$ if and only if
$$\lambda\in\bigcap_{k\in A}\left(\bigcup_{i=1}^{n-2}\{k-\lambda_i+i,k+2n-4+\lambda_i-i\}\cup\{2k+2n-3\}\right)$$
where 
$$A=\{N|N\in\mathbb{Z},\  -\lambda_1-2n+4\leq N\leq \lambda+\lambda_1-1, \ N\neq\lambda_i-i-1,\ N\neq-\lambda_i-2n+3+i\ \text{for all}\ i\}
$$
\end{theorem}

\subsection{On the minimal orbit closure of $\mathfrak{so}_{2n+1}$}
Now let's do similar work on $SO(2n+1)$. We use the quadratic form $J=\begin{bmatrix}0 & 0 & I'\\ 0 & 1 & 0 \\I'^T & 0 & 0\end{bmatrix}$ where $I'$ is the permutation matrix $\begin{bmatrix}
    0 & \operatorname{Id}_2\\
    \operatorname{Id}_{n-2} & 0 
\end{bmatrix}$. A typical element in the minimal orbit of $SO(2n+1)=SO(J)$ is $E_{1,2n+1}-E_{2,2n}$. Denote$J''=\begin{bmatrix} 0 &  0 &\operatorname{Id}_{n-2}\\
0 & 1 & 0\\
\operatorname{Id}_{n-2} & 0 & 0
\end{bmatrix} $, the stabilizer $H=\operatorname{Stab}_{SO(2n+1)}(E_{1,2n+1}-E_{2,2n})$ contains matrices in the form of
$$\begin{bmatrix}
U & P & B\\
0 & A & Q\\
0 & 0 & (U^{-1})^T
\end{bmatrix}
$$
where
$$U\in SL(2), \ A\in SO(J''),\ U^{-1}P+Q^TJ''A=0,\ U^{-1}B+(U^{-1}B)^T+Q^TJ''Q=0$$
\ 
The levi factor is $H/U(H)=SL(2)\times SO(J'')\cong SL(2)\times SO(2n-3)$.

Again consider the following copy of $\mathbb{C}^\times$ in $SO(2n+1)$: $S=\{\operatorname{diag}(x\operatorname{Id}_2, \operatorname{Id}_{2n-3}, x^{-1}\operatorname{Id}_2)|x\in\mathbb{C}^{\times}\}$. One checks easily that $S\cap H=\{\operatorname{diag}(\pm\operatorname{Id}_2, \operatorname{Id}_{2n-3}, \pm\operatorname{Id}_2)|x\in\mathbb{C}^{\times}\}\cong \mu_2$, $S$ normalize $H$ and $P=SH$ is parabolic. Indeed, with the induced action of $P$ on the projectivization of the minimal nilpotent orbit $\mathcal{O}_{\min}(\mathfrak{so}_{2n+1})$, $P$ is the stabilizer of the line $\mathbb{C}(E_{1,2n+1}-E_{2,2n})$.

Write $S=\operatorname{Spec}\mathbb{C}[t,t^{-1}]$. For an irreducible representation $\rho:H\rightarrow GL(V)$ of $H$, we consider the corresponding vector bundle $\mathcal{E}_V$ on $G/H$. Push it forward along the quotient map $q:G/H\rightarrow G/P$, as before we get 
$$(q_*\mathcal{E}_V)_{[1]}=\begin{cases}
    \displaystyle \bigoplus_{i\in\mathbb{Z}} \mathbb{C}t^{2i}\otimes V^\vee, & \rho|_{\mu_2}\  \text{is trivial}\\
    \displaystyle \bigoplus_{i\in\mathbb{Z}} \mathbb{C}t^{2i+1}\otimes V^\vee, & \rho|_{\mu_2} \ \text{is nontrivial}
\end{cases}
$$
as $P$ representations, and each summand $\mathbb{C}t^j\otimes V$ is irreducible. Now any irreducible representation of $H$ is an irreducible representation of $H/U(H)\cong SL(2)\times SO(2n-3)$, which is can be classified by $W\boxtimes W'$ where $W$ is an irreducible $SL(2)$ module and $W'$ is an irreducible $SO(2n-3)$ module.

Irreducible $SO(2n-3)$ representations are classified by dominant integral weights $(\lambda_1,\lambda_2,...,\lambda_{n-2})$ where $\lambda_i\in\mathbb{Z}$, $\lambda_1\geq\lambda_2\geq...\geq\lambda_{n-3}\geq \lambda_{n-2}\geq 0$. While the $SL(2)$ irreducible representations classified by nonnegative integers, all representations of $V$ can be classified by weight in the lattice $\mathbb{Z}^{n-1}$. For a weight $(\lambda, \lambda_1,\lambda_2,...,\lambda_{n-2})$ where $\lambda\geq 0,\  \lambda_1\geq\lambda_2\geq...\geq\lambda_{n-3}\geq \lambda_{n-2}\geq 0$.

Now consider the following maximal torus in $P$
$$T=\{\operatorname{diag}(t_1,...,t_n,1, t_3^{-1},...,t_n^{-1},t_1^{-1},t_2^{-1})| t_1,...,t_n\in\mathbb{C}^{\times}\}
$$
Replace $t_1$ with $ax$ and $t_2$ with $a^{-1}x$, the action of $T$ on the lowest weight vector $\mathbb{C}t^{-j}\otimes v^*$ is given by 
$$\operatorname{diag}(ax,a^{-1}x,t_3,....,t_n,1,t_3^{-1},...,t_n^{-1},a^{-1}x^{-1},ax^{-1}) (t^{-j}\otimes v^*)=x^{j}a^{-\lambda}t_3^{-\lambda_1}t_4^{-\lambda_2}...t_n^{-\lambda_{n-2}}(t^{-j}\otimes v^*)
$$
While $x^{j}a^{-\lambda}=(ax)^{-\frac{\lambda-j}{2}}(a^{-1}x)^{-\frac{-\lambda-j}{2}}$, we see that $\mathbb{C}t^j\otimes V$ is the dual representation of the irreducible representation with the highest weight $(\frac{\lambda-j}{2},\frac{-\lambda-j}
{2},\lambda_1,...,\lambda_{n-2})$. So 
$$q_*\mathcal{E}_V=\bigoplus_{j\equiv\lambda(\operatorname{mod}2)}\mathcal{E}(\frac{\lambda-j}{2},\frac{-\lambda-j}
{2},\lambda_1,...,\lambda_{n-2})
$$

Now the dimension of minimal orbit of $\mathfrak{so}_{2n+1}$ is $4n-4$. By Proposition \ref{localseqMCM}, the pushforward of $\mathcal{E}_V$ on minimal orbit closure is maximal Cohen-Macaulay if and only if $$H^m(\mathcal{E}_V,G/H)=\bigoplus_{j\equiv\lambda (\operatorname{mod} 2)}H^m(\mathcal{E}(\frac{\lambda-j}{2},\frac{-\lambda-j}
{2},\lambda_1,...,\lambda_{n-2}),G/P)=0
$$ 
for $1\leq m \leq 4n-6$. By BWB theorem \ref{BWBparabolic}, this is equivalent to that $(\frac{\lambda-j}{2},\frac{-\lambda-j}
{2},\lambda_1,...,\lambda_{n-2})+\rho$ is either singular, or the Weyl group element $w$ sending it to a dominant weight has $l(w)\in\{0,4n-5\}$. Here $\rho$ for $\mathfrak{so}_{2n+1}$ is $(n-\frac{1}{2},n-\frac{3}{2},...,\frac{1}{2})$. Let $k=\frac{\lambda-j}{2}$, then $\frac{-\lambda-j}{2}=k-\lambda$, when $j$ goes over all integers $\equiv \lambda (\operatorname{mod} 2)$, $k$ goes through all integers, so the weight we are considering is $\Lambda=(k+n-\frac{1}{2},k-\lambda+n-\frac{3}{2},\lambda_1+n-\frac{5}{2},...,\lambda_{n-2}+\frac{1}{2})$. 

Since $\lambda\geq 0$, we have $k+n-\frac{1}{2}> k-\lambda+n-\frac{3}{2}$. Also we see by $\lambda_1\geq\lambda_2\geq...\geq\lambda_{n-3}\geq \lambda_{n-2}\geq 0$ that $\lambda_1+n-\frac{5}{2}>...>\lambda_{n-2}+\frac{1}{2}>0$. So when $\Lambda$ is nonsingular, the length of $w$ can be computed by:
\begin{align}\notag
l(w)&=\{\alpha\in\Phi^+|\langle\Lambda,\alpha^\vee\rangle<0\}\\\notag
&=\#\{i|k+n-\frac{1}{2}<\lambda_i+n-\frac{3}{2}-i\}+\#\{i|k-\lambda+n-\frac{3}{2}<\lambda_i+n-\frac{3}{2}-i\}\\\notag
&\quad+\ \#\{i|k+n-\frac{1}{2}<-\lambda_i-n+\frac{3}{2}+i\}+\#\{i|k-\lambda+n-\frac{3}{2}<-\lambda_i-n+\frac{3}{2}+i\}\\\notag
&\quad + 1_{k+n-\frac{1}{2}+k-\lambda+n-\frac{3}{2}<0}+\ 1_{k+n-\frac{1}{2}<0}+1_{k-\lambda+n-\frac{3}{2}<0}\\\notag
&\leq 4\times(n-2)+1+1+1\\\notag
&=4n-5
\end{align}
The equal sign holds if and only if all four sets have cardinality $n-2$ and $k+n-\frac{1}{2}<0$. This condition is equivalent to $k+n-\frac{1}{2}<-\lambda_1-n+\frac{5}{2}$, i.e., $k<-\lambda_1-2n+3$. On the other hand, we have an apparent equivalence that $l(w)=0$ if and only if $k-\lambda+n-\frac{3}{2}>\lambda_1+n-\frac{5}{2}$, i.e., $k>\lambda+\lambda_1-1$. Hence, when $\Lambda$ is nonsingular, $l(w)=0$ or $4n-5$ if and only if $k>\lambda+\lambda_1-1$ or $k<-\lambda_1-2n+3$ respectively. In order to satisfy the cohomology vanishing condition, for $-\lambda_1-2n+3\leq k\leq\lambda+\lambda_1-1$, $\Lambda$ should be singular.

Define the set of gaps
$$A=\{N|N\in\mathbb{Z}, -\lambda_1-2n+3\leq N\leq\lambda+\lambda_1-1, N\neq \lambda_i-1-i,\ N\neq-\lambda_i-2n+2+i\ \text{for all}\ i\}
$$
When $k\in A$, $k+n-\frac{1}{2}\neq\pm(\lambda_i+n-\frac{3}{2}-i)$, and obviously $k+n-\frac{1}{2}\neq 0$ because its not an integer. Then to make $\Lambda$ singular, we should have $k-\lambda+n-\frac{3}{2}=\pm(\lambda_i+n-\frac{3}{2}-i)$ for some $i$ or $k-\lambda+n-\frac{3}{2}+k+n-\frac{1}{2}=0$, i.e., $k\in\{\lambda+\lambda_i-i,\lambda-\lambda_i-2n+3+i\}$ for some $i$ or $2k-\lambda+2n-2=0$. So $\lambda=2k+2n-2$ or it is in $\{k-\lambda_i+i,k+\lambda_i+2n-3-i\}$ for each $k\in A$ and some $1\leq i\leq n-2$. 

\begin{theorem}\label{Btypetheorem}
Suppose $\mathcal{O}_{\min}$ is the minimal orbit in $\mathfrak{so}_{2n+1}$ $(n\geq 3)$ with respect to the conjugation action of $SO(2n+1)$ and $H$ is the stabilizer. An equivariant bundle $\mathcal{E}_V=SO(2n+1)\times_H V^\vee$ corresponding to an irreducible representation $V$ of $H$ is determined by a dominant weight $(\lambda,\lambda_1,\lambda_2...\lambda_{n-2})$ with $\lambda\geq 0$ and $\lambda_1\geq...\geq\lambda_{n-3}\geq\lambda_{n-2}\geq0$ in $H$'s lattice $\mathbb{Z}^{n}$. The pushforward $i_*\mathcal{E}_V$ is maximal Cohen-Macaulay on the orbit closure $\overline{\mathcal{O}_{\min}}$ if and only if 
for any element $k$ in the following set
$$\lambda\in\bigcap_{k\in A}\left(\bigcup_{i=1}^{n-2}\{k-\lambda_i+i,k+\lambda_i+2n-3-i\}\cup\{2k+2n-2\}\right)$$
where 
$$A=\{N|N\in\mathbb{Z}, -\lambda_1-2n+3\leq N\leq\lambda+\lambda_1-1, N\neq \lambda_i-1-i,\ N\neq-\lambda_i-2n+2+i\ \text{for all}\ i\}
$$
\end{theorem}

\section{Highest-root vector orbit and exceptional cases}

In order to generalize our results to exceptional types of simple Lie algebras, we need to understand the story in the previous chapters in a coordinate-free manner. All constructions in this chapter on  simple Lie algebras are standard and can be found in \cite{HumphreysLieAlgebras}.

\subsection{Highest-root vector orbit method and $\operatorname{Spin}$-equivariant sheaves}\label{abstractanalysis}
Suppose $\mathfrak{g}$ is a simple Lie algebra over $\mathbb{C}$ and $G$ is the corresponding simply connected Lie group. Fix a maximal toral subalgebra $\mathfrak{t}$ in $\mathfrak{g}$, we have a Cartan decomposition
$$\mathfrak{g}=\mathfrak{t}\oplus\bigoplus_{\alpha\in\Phi}\mathfrak{g}_\alpha
$$
where $\Phi$ is the irreducible root system of $\mathfrak{g}$. Suppose $\Delta=\{\alpha_1,\alpha_2,...,\alpha_n\}$ is a basis of $\Phi$ and $\theta$ is the highest root with respect to this basis. There exist $e_\theta, f_\theta\in\mathfrak{g}$ and $h_\theta\in\mathfrak{t}$ such that $\mathfrak{g}_\theta=\mathbb{C}e_\theta$, $\mathfrak{g}_{-\theta}=\mathbb{C}f_\theta$ and $\mathbb{C}\langle e_\theta,f_\theta,h_\theta\rangle\cong\mathfrak{sl}(2,\mathbb{C})$, where 
$[h_\theta,e_\theta]=2e_\theta$, $[h_\theta,f_\theta]=-2f_\theta$ and $[e_\theta,f_\theta]=h_\theta$. In addition, suppose $\kappa$ is the Killing form, we know that $h_\theta=\frac{2t_\theta}{\kappa(t_\theta,t_\theta)}$ where $t_\theta$ is the unique element in $\mathfrak{t}$ such that $\kappa(t_\theta,h)=\theta(h)$ for any $h\in\mathfrak{t}$.

Now we consider the action of $h_\theta$ on $\mathfrak{g}$. Since $h_\theta\in\mathfrak{t}$, we have $[h_\theta,\mathfrak{t}]=0$. For $x_\alpha\in\mathfrak{g}_\alpha$, we have $[h_\theta,x_\alpha]=\alpha(h_\theta)x_\alpha$. While $\theta$ is the highest root, it is a long root, so we get that 
$$|\alpha(h_\theta)|=\frac{2|\kappa(t_\alpha,t_\theta)|}{(t_\theta,t_\theta)}=\frac{2|(\alpha,\theta)|}{(\theta,\theta)}=2\frac{||\alpha||}{||\theta||}|\cos\langle\alpha,\theta\rangle|\leq 2
$$
with the equality holds only if $\cos\langle\alpha,\theta\rangle=\pm 1$, i.e., $\alpha=\pm\theta$. So we have a decomposition of $\mathfrak{g}$ with respect to the action of $h_\theta$:
$$\mathfrak{g}=\mathbb{C}f_\theta\oplus\mathfrak{g}_{-1}\oplus\mathfrak{g}_0\oplus\mathfrak{g}_1\oplus\mathbb{C}e_\theta
$$
where $\mathfrak{g}_{\pm 1}=\bigoplus_{\langle\alpha,\theta\rangle=\pm 1}\mathfrak{g}_\alpha$ and $\mathfrak{g}_0=\mathfrak{t}\oplus\bigoplus_{\langle\alpha,\theta\rangle=0}\mathfrak{g}_\alpha$. Here we denote $\langle\alpha,\theta\rangle=\frac{2(\alpha,\theta)}{(\theta,\theta)}$.

Now for $\alpha\in\Phi$, if $\langle\alpha,\theta\rangle=0$, we see that $(\alpha,\theta)=0$, so $\langle\alpha,\theta\rangle=0$ as well. Then we know that the $\alpha$-string through $\theta$ is symmetric: $\theta-q\alpha,...,\theta,...,\theta+q\alpha$. If $\alpha\in \Phi_+$, we know that $\theta+\alpha$ is not a root, then neither is $\theta-\alpha$. Similarly, $\alpha\in\Phi_-$,  $\theta-\alpha$ is not a root, then neither is $\theta+\alpha$. Hence, we can conclude that none of $\pm(\theta\pm\alpha)$ is a root. Therefore $[\mathfrak{g}_{\pm\theta},\mathfrak{g}_{\pm\alpha}]=0$. On the other hand, for $h\in\mathfrak{t}$, $[h,e_\theta]=\theta(h)e_\theta$, $[h,f_\theta]=-\theta(h)f_\theta$, so for $h\in\ker\theta$, we have $[h,e_\theta]=[h,f_\theta]=0$. So $\mathfrak{g}_0$ can be decomposed into $\mathbb{C}h_\theta\oplus\mathfrak{g}^\natural$ where $\mathfrak{g}^\natural=\ker\theta\oplus\bigoplus_{\langle\alpha,\theta\rangle=0}\mathfrak{g}_\alpha$ is the commutator of $\mathfrak{sl}_2(\theta)=\mathbb{C}\langle e_\theta,f_\theta,h_\theta\rangle$. 

Now let's discuss the commutator of $e_\theta$, which we denote by $\mathfrak{c}(e_\theta)$. Apparently, $\mathfrak{g}^\natural$ is in this commutator. On the other hand, for $\alpha\in\Phi$, $\langle\alpha,\theta\rangle>0$, we see that $\alpha\in\Phi_+$ and that $\theta+\alpha$ is not a root, so $[\mathfrak{g}_\alpha,\mathfrak{g}_\theta]=0$. Therefore, $\mathfrak{g}_1\oplus\mathbb{C}e_\theta\in\mathfrak{c}(e_\theta)$. On the other hand, for those $\alpha$ such that $\langle\alpha,\theta\rangle<0$. we see that $\theta-\langle\alpha,\theta\rangle\alpha$ is a root. While $\theta+\alpha$ is on the $\theta$-string between $\theta$ and $\theta-\langle\alpha,\theta\rangle\alpha$, $\theta+\alpha$ is also a root. Hence,  $[\mathfrak{g}_\alpha,\mathfrak{g}_\theta]=\mathfrak{g}_{\alpha+\theta}\neq0$. For different such $\alpha$, the roots $\alpha+\theta$ are different; hence, the vectors in those $\mathfrak{g}_{\theta+\alpha}$ are linearly independent with each other and also with $e_\theta$. Hence, it's easy to see that $\mathfrak{c}(e_\theta)=\mathfrak{g}^\natural\oplus \mathfrak{g}_1\oplus\mathbb{C}e_\theta$.

If one only pays attention to those elements fixing the $1$-dimensional space $\mathbb{C}e_\theta$, then besides the commutator above, $h_\theta$ is also obtained. One can therefore see that $\mathfrak{c}(\mathbb{C}e_\theta)=\mathbb{C}h_\theta\oplus\mathfrak{c}(e_\theta)=\mathfrak{g}_0\oplus\mathfrak{g}_1\oplus\mathfrak{g}_2$. One can see easily that $\mathfrak{c}(\mathbb{C}e_\theta)$ is parabolic because each positive root has nonnegative inner product with $\theta$.

It's easy to observe that $\mathfrak{g}_1\oplus\mathbb{C}e_\theta$ is nilpotent and normal in both $\mathfrak{c}(e_\theta)$ and $\mathfrak{c}(\mathbb{C}e_\theta)$. One can also easily check that $\mathfrak{g}^\natural$ (or $\mathfrak{g}_0$) is reductive, and the root system of $[\mathfrak{g}^\natural,\mathfrak{g}^\natural]$ is generated by those simple roots that are orthogonal to $\theta$. The Dynkin diagram of $[\mathfrak{g}^\natural,\mathfrak{g}^\natural]$ is obtained by taking the corresponding affine Dynkin diagram of $\mathfrak{g}$ and removing the extra point and those points connected to the extra point, as classified in the following table:
\begin{table}[htbp]
\centering
\renewcommand{\arraystretch}{1.6}
\setlength{\tabcolsep}{9pt}
\begin{tabular}{|c|c|c|c|c|c|c|c|c|c|}
\hline
\textbf{Type of $\mathfrak{g}$}
& $A_n$ & $B_n$ & $C_n$ & $D_n$ & $E_6$ & $E_7$ & $E_8$ & $F_4$ & $G_2$ \\ \hline
\textbf{Type of $[\mathfrak{g}^{\natural},\mathfrak{g}^{\natural}]$}
& $A_{n-2}$ & $A_1\cup B_{n-2}$ & $C_{n-2}$ & $A_1\cup D_{n-2}$ & $A_5$ & $D_6$ & $E_7$ & $C_3$ & $A_1$ \\ \hline
\end{tabular}
\end{table}

At the group level, what happens is that under the adjoint action of $G$ on $\mathfrak{g}$, the stabilizer $H$ of $e_\theta$ has Lie algebra $\mathfrak{c}(\theta)=\mathfrak{g}^\natural\oplus \mathfrak{g}_1\oplus\mathbb{C}e_\theta$, while the stabilizer $P$ of the line $\mathbb{C}e_\theta$ has Lie algebra $\mathfrak{c}(\mathbb{C}e_\theta)=\mathbb{C}h_\theta\oplus\mathfrak{c}(e_\theta)=\mathfrak{g}_0\oplus\mathfrak{g}_1\oplus\mathfrak{g}_2$. Let $S=\theta^\vee(\mathbb{G}_m)=\exp(\mathbb{C}h_\theta)$. Since $[h_\theta,e_\theta]=2e_\theta$, we see that for $s\in S\cong\mathbb{G}_m$, $\operatorname{Ad}(s).e_\theta=s^2e_\theta$. Now consider the group morphism $\varphi: P\rightarrow\mathbb{G}_m$ induced by the action of $P$ on $\mathbb{C}e_\theta$. We see that $\ker\varphi=H$, so $H$ is normal in $P$, and that $S\cap H= S\cap\ker\varphi=\{\pm1\}\cong\mu_2$.

Suppose $L$ is the Levi factor of $H$, then $L$ has Lie algebra $\mathfrak{g}^\natural$ and $[L,L]$ has Lie algebra $[\mathfrak{g}^\natural,\mathfrak{g}^\natural]$. While one can also realize $L$ as the Levi factor of parabolic subgroup $P$, by \cite{ConradReductiveGroupsOverFields} Corollary 9.5.11, since $G$ is simply connected, $[L,L]$ is also simply connected. Then the irreducible representation of $L$ (hence of $H$, since $L$ is the Levi factor of $H$) are classified by dominant weights in $X^*(T\cap H)$, where $T$ is the torus with respect to $\mathfrak{t}$. 

Since $\mathfrak{g}^{\natural}$ commutes with $h_\theta$, we see that $L$ commutes with $S$. From previous paragraphs, it is also easy to check that $(T\cap H)\cdot S=T$ and $(T\cap H)\cap S=H\cap S=\mu_2$.

Since $\operatorname{Lie}(P)=\mathfrak{c}(\mathbb{C}e_\theta)$ is parabolic, we see that $P$ is parabolic as well. Therefore $G/P$ is proper. Now $G/P$ is the orbit of $[\mathbb{C}e_\theta]$ in $\mathbb{P}(\mathfrak{g})$, we see that $G\cdot[\mathbb{C}e_\theta]$ is closed in $\mathbb{P}(\mathfrak{g})$, so $Ge_\theta$ is the minimal nonzero nilpotent orbit $\mathcal{O}_{\min}$. 

In summary, for any simple Lie group $G$, the story we have discussed for $A,B,C,D$ types is:
\begin{enumerate}
    \item Write the minimal orbit  as $\mathcal{O}_{\min}=G/H$, where $H$ is the stabilizer of the adjoint action of a highest-root vector. 
    \item Take an appropriate maximal torus $T$ and find $S\cong\mathbb{G}_m$ in $T$ such that $S\cdot H=P$ is parabolic, $S$ normalizes $H$, $S\cap H=S\cap(T\cap H)\cong\mu_2$ and $S$ commutes with $L$, the Levi factor of $H$.
    \item Under the conditions above, $\mathcal{O}_{\min}=G/H\rightarrow G/P$ is an affine map with fibers $S/\mu_2$, so for an equivariant sheaf $\mathcal{E}_V=G\times_H V^\vee$ on $G/H$ with respect to an $H$-representation $V$, we have $H^i(\mathcal{E}_V,G/H)=H^i(q_*\mathcal{E}_V,G/P)$.
    \item If $V$ is irreducible, $V$ is also an irreducible representation of $L$ and can be classified by a dominant weight $\lambda$ in $X^*(T\cap H)$. Since $L$ commutes with $S$ the $P$-representation with respect to $q_*\mathcal{E}_V$ splits into all possible irreducible representations of $P$ whose restrictions to $H$ is $V^\vee$. Those are irreducible $H$-representations with an $S$ action that is compatible with $S\cap H\cong \mu_2$ action. Those representations can be represented using characters of $P$, so the application of BWB theorem tells us what $H^i(q_*\mathcal{E}_V,G/P)$ are.
    \item Applying Proposition \ref{localseqMCM}, find the necessary and sufficient condition for $\lambda$ such that we have vanishing $H^i(q_*\mathcal{E}_V,G/P)$ for $1\leq i\leq\dim\mathcal{O}_{\min}-2$. Those are the characters giving equivariant maximal Cohen-Macaulay sheaves on $\overline{\mathcal{O}_{\min}}$.
\end{enumerate}

With this story in mind, the first thing we can do is to generalize the story in the previous chapter to $\operatorname{Spin}(2n)$ and $\operatorname{Spin}(2n+1)$.

\begin{theorem}
Suppose $\mathcal{O}_{\min}$ is the minimal orbit in $\mathfrak{so}_{2n}$ $(n\geq 3)$ with respect to the conjugation action of $\operatorname{Spin}(2n)$, and $H$ is the stabilizer. An equivariant bundle $\mathcal{E}_V=\operatorname{Spin}(2n)\times_H V^\vee$ corresponding to an irreducible representation $V$ of $H$ is determined by a dominant weight $(\lambda,\lambda_1,\lambda_2...\lambda_{n-2})$ with $\lambda\geq 0$ and $\lambda_1\geq...\geq\lambda_{n-3}\geq|\lambda_{n-2}|$ in $H$'s lattice $\mathbb{Z}^{n}+(0,\frac{1}{2},...,\frac{1}{2})\mathbb{Z}$. The pushforward $i_*\mathcal{E}_V$ is maximal Cohen-Macaulay on the orbit closure $\overline{\mathcal{O}_{\min}}$ if and only if
$$\lambda\in\bigcap_{k\in A}\left(\bigcup_{i=1}^{n-2}\{k-\lambda_i+i,k+2n-4+\lambda_i-i\}\cup\{2k+2n-3\}\right)$$
where 
$$A=\{N|N\in\mathbb{Z}+\lambda_1,\  -\lambda_1-2n+4\leq N\leq \lambda+\lambda_1-1, \ N\neq\lambda_i-i-1,\ N\neq-\lambda_i-2n+3+i\ \text{for all}\ i\}
$$
\end{theorem}

\begin{theorem}
Suppose $\mathcal{O}_{\min}$ is the minimal orbit in $\mathfrak{so}_{2n+1}$ $(n\geq 3)$ with respect to the conjugation action of $\operatorname{Spin}(2n+1)$ and $H$ is the stabilizer. An equivariant bundle $\mathcal{E}_V=\operatorname{Spin}(2n+1)\times_H V^\vee$ corresponding to an irreducible representation $V$ of $H$ is determined by a dominant weight $(\lambda,\lambda_1,\lambda_2...\lambda_{n-2})$ with $\lambda\geq 0$ and $\lambda_1\geq...\geq\lambda_{n-3}\geq\lambda_{n-2}\geq0$ in $H$'s lattice $\mathbb{Z}^{n}+(0,\frac{1}{2},...,\frac{1}{2})\mathbb{Z}$. The pushforward $i_*\mathcal{E}_V$ is maximal Cohen-Macaulay on the orbit closure $\overline{\mathcal{O}_{\min}}$ if and only if 
for any element $k$ in the following set
$$\lambda\in\bigcap_{k\in A}\left(\bigcup_{i=1}^{n-2}\{k-\lambda_i+i,k+\lambda_i+2n-3-i\}\cup\{2k+2n-2,k+n-\frac{3}{2}\}\right)$$
where 
$$A=\{N|N\in\mathbb{Z}+\lambda_1, -\lambda_1-2n+3\leq N\leq\lambda+\lambda_1-1, N\notin\bigcup_{i=1}^{n-2} \{\lambda_i-1-i,-\lambda_i-2n+2+i\},N\neq \frac{1}{2}-n\}
$$
\end{theorem}

\subsection{Exceptional cases}\label{exceptional}

For exceptional cases, we know the types of semisimple Lie groups $[L,L]$ from the previous chapter. What is left for us to do is to carefully match the characters of $T\cap H$ and the characters of $T$ such that we can obtain the numerical results. 

We will use the language from the previous section: Suppose $G$ is a simple group of exceptional type with Lie algebra $\mathfrak{g}$, $T$ is a maximal torus with Lie algebra $t$ and $\Phi$ is the corresponding root system. Let's also choose a basis $\Delta$ and suppose $\theta$ is the highest root with the highest root vector $e_\theta$ and half positive root sum $\rho$. We also suppose that $H$
is the stabilizer of $e_\theta$ with respect to the adjoint action of $G$, $P$ is the stabilizer of $\mathbb{C}e_\theta$, $L$ is the Levi factor of $P$. Finally, let's denote $S=\theta^\vee(\mathbb{G}_m)$.

Let's start with a lemma that shows the semisimplicity and connectedness of the Levi factor of $H$.

\begin{lemma}\label{exceptionalsimple}
Suppose $U(H)$ is the unipotent radical of $H$, then $H/U(H)=[L,L]$, it is connected, simply connected, and simple with root system $\Phi^\natural$ by removing extra vertex and points connected to it in the affine Dynkin diagram of $G$.
\end{lemma}
\begin{proof}
The parabolic group $P$ acts on $\mathbb{C}e_\theta$ by character $\chi$, i.e., $g.e_\theta=\chi(g)e_\theta$ for $g\in P$. Then $H=\ker\chi$. Now $P$ has Levi decomposition $P=L\ltimes U$. By the property of unipotent groups, $U\in\ker\chi= H$. Therefore, we see that the Levi decomposition of $H$ is $H=\ker\chi|_L\ltimes U$, so $H/U(H)=\ker\chi|_L$. So we need to show $\ker\chi|_L=[L,L]$.

As discussed in the previous section, $\mathfrak{g}$ has decomposition
$$
\mathfrak{g}=\mathbb{C}f_\theta\oplus\mathfrak{g}_{-1}\oplus\mathfrak{g}^\natural\oplus\mathbb{C}h_\theta\oplus\mathfrak{g}_1\oplus\mathbb{C}e_\theta
$$
and $\operatorname{Lie}(L)=\mathfrak{g}^\natural\oplus\mathbb{C}h_\theta$. Now check case by case, we know that the root system $\Phi^\natural$ for $\mathfrak{g}^\natural$ satisfies that $\operatorname{rank}\Phi^\natural=\operatorname{rank}\Phi-1$. While the rank of the maximal toral subalgebra of $\mathfrak{g}^\natural$ is also one less than the rank of the maximal toral subalgebra of $\mathfrak{g}$ (since the only difference is $\mathbb{C}h_\theta$), we see that $\mathfrak{g}^\natural$ is semisimple with root system $\Phi^\natural$. Then $\operatorname{Lie}([L,L])=\mathfrak{g}^\natural$. As we discussed in the previous section, by \cite{ConradReductiveGroupsOverFields} Corollary 9.5.11, $[L,L]$ is connected, simply connected.

Since $S=\theta^\vee(\mathbb{G}_m)=\exp(\mathbb{C}h_\theta)$, we see that $S\cdot[L,L]=L$. Now $[L,L]\subseteq\ker\chi$, and we can compute
$$\operatorname{Ad}(\theta^\vee(t))e_\theta=t^{\langle\theta^\vee,\theta\rangle}e_\theta=t^2e_\theta
$$
So $\theta^\vee(t)\in\ker\chi$ if and only if $t=\pm 1$. So $H/U(H)=[L,L]$ if and only if $\theta^\vee(-1)\in[L,L]$. Now the maximal torus in $[L,L]$ is generated by $\alpha^\vee(\mathbb{G}_m)$ for $\alpha\in\Delta\cap\Phi^\natural$. Suppose $\Delta=\alpha_1,...,\alpha_n$ and $I$ is the subset of indices $i$ such that $\alpha_i\in\Phi^\natural$. If we denote the coroot lattice of $\Phi$ by $Q^\vee$, we see that as long as 
$\theta^\vee\equiv\sum_{i\in I}c_i\alpha_i^\vee(\operatorname{mod} 2Q^\vee)$ for some integers $c_i,i\in I$, we have $\theta^\vee(-1)=\prod_{i\in I}\alpha_i^\vee(-1)^{c_i}\in[L,L]$. This is satisfied for exceptional case Lie algebras. For example, in type $E_6$, $\theta^\vee=\alpha_1^\vee+2\alpha_2^\vee+2\alpha_3^\vee+3\alpha_4^\vee+2\alpha_5^\vee+\alpha_6^\vee$ and $I=\{1,3,4,5,6\}$, so $\theta^\vee\equiv\alpha_1^\vee+2\alpha_3^\vee+3\alpha_4^\vee+2\alpha_5^\vee+\alpha_6^\vee (\operatorname{mod} Q^\vee)$. Indeed, in the expression of $\theta^\vee$ in terms of coroot basis, the coordinate with respect to the basis coroot which is connected to $\theta$ in the affine Dynkin diagram is always 2. 
\end{proof}

Notice that there are two properties 
for exceptional cases that are not satisfied for $A,B,C,D$ types. Firstly,  $\operatorname{rank}\Phi^\natural=\operatorname{rank}\Phi-1$ fails for $A_n$ types, so $H/U(H)$ for $A_n$ type is a double cover of $GL_{n-1}$ which is not semisimple. In addition, the last property in the proof fails for $C_n$, so $H/U(H)$ for $C_n$ type is the product of $\mu_2$ and a $Sp(2n-2)$.

Now an irreducible $[L,L]$-representation (equivalently, an irreducible $H$-representation) $V_\lambda$ is determined by a dominant weight $\lambda$ in the weight lattice of $\Phi^\natural$. This gives an equivariant bundle $\mathcal{E}_V$ on $G/H$. Now $P=SH$, $S\cap H=\mu_2$, so the quotient map $q:G/H\rightarrow G/P$ is affine with fibers $S/\mu_2$. When pushing forward along $q$, since $S$ commutes with $L$ and hence $[L,L]$, the vector bundle splits into irreducible $P$-equivariant bundles as follows
$$q_*\mathcal{E}_{V_\lambda}\cong\bigoplus_{j\in J}G\times_H(V_\lambda^\vee\otimes\mathbb{C}t^{-j})\cong\bigoplus_{j\in J}\mathcal{E}_{\lambda+\frac{j}{2}\theta}
$$
where $J$ is the subset of $\mathbb{Z}$ consisting of integers that make the weight $\lambda+\frac{j}{2}\theta$ integral. This is because one can check that (using the notations in the proof of Lemma \ref{exceptionalsimple})
$$\langle\lambda+\frac{j}{2}\theta,\alpha_i^\vee\rangle=\langle\lambda,\alpha_i^\vee\rangle,\ i\in I;\quad \langle\lambda+\frac{j}{2}\theta,\theta^\vee\rangle=\frac{j}{2}\langle\theta,\theta^\vee\rangle=j
$$
So the character $\lambda+\frac{j}{2}\theta$ is the same as $\lambda$ restricted to the maximal torus of $[L,L]$ and is of weight $j$ when restricted to $S$. This is exactly how the maximal torus in $P$ acts on $V_\lambda\otimes\mathbb{C}t^j$. Therefore, we have
$$H^\ast (\mathcal{E}_{V_\lambda},G/H)=\bigoplus_{j\in J}H^\ast(\mathcal{E}_{\lambda+\frac{j}{2}\theta},G/P)
$$
Now suppose $d=\dim\mathcal{O}_{\min}$. By Proposition \ref{localseqMCM} and the BWB theorem, the statement that $i_*\mathcal{E}_{V_\lambda}$ is maximal Cohen-Macaulay is equivalent to that for any $j\in J$, $\lambda+\frac{j}{2}\theta+\rho$ is either singular or $l(w)\in\{0,d-1\}$ for the unique Weyl group element $w$ sending it $\lambda+\frac{j}{2}\theta+\rho$ to a dominant weight.

For integral nonsingular weight $\Lambda=\lambda+\frac{j}{2}\theta+\rho$, the length of $w$ can be computed as $\#\{\beta\in\Phi^+|\langle\Lambda,\beta^\vee\rangle<0\}$. Now since $\lambda$ is already dominant integral in the weight space of $\Phi^\natural$, we see that $\langle\Lambda,\beta\rangle=0$ for $\beta\in (\Phi^\natural)^+$, so essentially we have 
$$l(w)=\#\{\beta\in\Phi^+|\langle\Lambda,\beta^\vee\rangle<0\}=\#\{\beta\in\Phi^+\setminus(\Phi^\natural)^+|\langle\Lambda,\beta^\vee\rangle<0\}
$$
Now $\#\Phi^+-\#(\Phi^\natural)^+=\dim G/P=d-1$, so we see that $l(w)=d-1$ if and only if $\langle\Lambda,\beta^\vee\rangle<0$ for every $\beta\in\Phi^+\setminus(\Phi^\natural)^+$. In addition, we can also derive from above that $l(w)=0$ if and only if $\langle\Lambda,\beta^\vee\rangle>0$ for every $\beta\in\Phi^+\setminus(\Phi^\natural)^+$. Hence, $i_*\mathcal{E}_{V_\lambda}$ is maximal Cohen-Macaulay on $\mathcal{O}_{\min}$ if and only if, for each $j$, the set  $A_j=\{\langle\lambda+\frac{j}{2}\theta+\rho,\beta^\vee\rangle| \beta\in\Phi^+\setminus(\Phi^\natural)^+\}$ either contains $0$ or the elements in it have the same sign. 

Since in the exceptional cases, the highest root vertex is connected to a single simple root vertex in the affine Dynkin diagram, there is only one index, which we denote by $m$, that is not in $I$. For any root $\beta$, we can write $\beta^\vee=c_m\alpha_m^\vee+\sum_{i\in I}c_i\alpha_i^\vee$. Suppose $\langle\lambda,\alpha_i^\vee\rangle=y_i$ for $i\in I$, then $\langle\Lambda,\alpha_i^\vee\rangle=\langle\lambda+\frac{j}{2}\theta+\rho,\alpha_i^\vee\rangle=y_i+1$, which is independent of $j$. Now suppose $k=\langle\Lambda,\alpha_m^\vee\rangle$, one sees that when $j$ goes over elements in $J$, $k$ actually goes through all integers. Now 
$$\langle\Lambda,\beta^\vee\rangle=\langle\Lambda,c_m\alpha_m^\vee+\sum_{i\in I}c_i\alpha_i^\vee\rangle=c_mk+\sum_{i\in I}c_iy_i
$$
We see that the condition is equivalent to that for each $k<0$, either there exists $\beta\in \Phi^+\setminus(\Phi^\natural)^+$ such that $k=-\frac{\sum_{i\in I}c_iy_i}{c_m}$, or $k<-M=-\max_{\beta\in \Phi^+\setminus(\Phi^\natural)^+}\{\frac{\sum_{i\in I}c_iy_i}{c_m}\}$. This is equivalent to checking whether
$$\{1,2,...,\lfloor M\rfloor\}\in\left\{\left.\frac{\sum_{i\in I}c_iy_i}{c_m}\right|\beta\in\Phi^+\setminus(\Phi^\natural)^+\right\}
$$
If this holds, then $\frac{\sum_{i\in I}c_iy_i}{c_m}<\#\Phi^+-\#(\Phi^\natural)^+=d-1$. Since $\lambda$ is a dominant weight in $\Phi^\natural$, $y_i$ are nonnegative integers, so this is actually a finite check.

\subsection{Classification of weights in exceptional cases}\label{exceptionalclassifications}
Now what is left for us to do is to check the condition at the end of the last section case by case. We directly show the results and omit the tedious computation process.

\ 

\textbf{$E_6$ case:}
We use the realization
\[
V=\{x\in \mathbb R^8\mid x_6=x_7=-x_8\}.
\]
Take simple roots
\[
\begin{aligned}
\alpha_1&=\frac12(e_1-e_2-e_3-e_4-e_5-e_6-e_7+e_8),&
\alpha_2&=e_1+e_2,\\
\alpha_3&=-e_1+e_2,&
\alpha_4&=-e_2+e_3,\\
\alpha_5&=-e_3+e_4,&
\alpha_6&=-e_4+e_5.
\end{aligned}
\]
The Dynkin diagram is
\[
\begin{array}{ccccccccccc}
&&&& \alpha_2 &&&&&&\\
&&&& | &&&&&&\\
\alpha_1&-&\alpha_3&-&\alpha_4&-&\alpha_5&-&\alpha_6.
\end{array}
\]
The highest root is
\[
\theta=\alpha_1+2\alpha_2+2\alpha_3+3\alpha_4+2\alpha_5+\alpha_6
=\frac12(e_1+e_2+e_3+e_4+e_5-e_6-e_7+e_8).
\]
Let
\[
\beta_1=\alpha_1,\qquad
\beta_2=\alpha_3,\qquad
\beta_3=\alpha_4,\qquad
\beta_4=\alpha_5,\qquad
\beta_5=\alpha_6.
\]
Then \(\beta_1,\dots,\beta_5\) form an \(A_5\)-system. Its fundamental weights
\(\varpi_i\in \operatorname{span}_{\mathbb R}\{\beta_1,\dots,\beta_5\}\) are
\[
(\varpi_i,\beta_j)=\delta_{ij}.
\]
Explicitly,
\[
\begin{aligned}
\varpi_1&=-\frac14 e_1+\frac14(e_2+e_3+e_4+e_5)-\frac5{12}e_6-\frac5{12}e_7+\frac5{12}e_8,\\
\varpi_2&=-\frac23e_1+\frac13(e_2+e_3+e_4+e_5)-\frac13e_6-\frac13e_7+\frac13e_8,\\
\varpi_3&=-\frac12(e_1+e_2)+\frac12(e_3+e_4+e_5)-\frac14e_6-\frac14e_7+\frac14e_8,\\
\varpi_4&=-\frac13(e_1+e_2+e_3)+\frac23(e_4+e_5)-\frac16e_6-\frac16e_7+\frac16e_8,\\
\varpi_5&=-\frac16(e_1+e_2+e_3+e_4)+\frac56e_5-\frac1{12}e_6-\frac1{12}e_7+\frac1{12}e_8.
\end{aligned}
\]
A dominant \(A_5\)-weight is
\[
\lambda=a_1\varpi_1+a_2\varpi_2+a_3\varpi_3+a_4\varpi_4+a_5\varpi_5,
\qquad a_i\in \mathbb Z_{\geq 0}.
\]
The weights satisfying the maximal Cohen-Macaulay condition are precisely those with
\[
a_3=0
\]
and with \((a_1,a_2,a_4,a_5)\) appearing in the following table:
\[
\begin{array}{c|c}
(a_1,a_2,a_4) & \text{allowed }a_5 \\ \hline
(0,0,0) & 0,\dots,5\\
(0,0,1) & 0,\dots,5\\
(0,0,2) & 0,\dots,5\\
(0,1,0) & 0,\dots,4\\
(0,2,0) & 0\\
(1,0,0) & 0,\dots,6\\
(1,0,1) & 0,\dots,6\\
(1,1,0) & 0,\dots,5\\
(1,2,0) & 0\\
(2,0,0) & 0,\dots,7
\end{array}
\qquad
\begin{array}{c|c}
(a_1,a_2,a_4) & \text{allowed }a_5 \\ \hline
(2,0,1) & 0,1,2\\
(2,1,0) & 0,\dots,6\\
(2,2,0) & 0\\
(3,0,0) & 0,1,2\\
(3,0,1) & 0,1,2,3\\
(3,1,0) & 0,1,3\\
(3,2,0) & 0\\
(4,0,0) & 0,1,2\\
(4,0,1) & 0,1,2\\
(4,1,0) & 0,1
\end{array}
\qquad
\begin{array}{c|c}
(a_1,a_2,a_4) & \text{allowed }a_5 \\ \hline
(4,2,0) & 0\\
(5,0,0) & 0,1,2\\
(5,0,1) & 1,2\\
(5,1,0) & 0,1\\
(5,2,0) & 0\\
(6,0,0) & 1,2\\
(6,0,1) & 2\\
(6,1,0) & 1\\
(7,0,0) & 2\\
\phantom{(7,0,0)} & \phantom{2}
\end{array}
\]
There are $97$ such weights.

\ 

\textbf{$E_7$ case:}
We realize \(E_7\) in the subspace
\[
V=\{x\in \mathbb R^8\mid x_7=-x_8\}.
\]
The roots are
\[
\Phi(E_7)
=
\{\pm e_i\pm e_j\mid 1\leq i<j\leq 6\}
\cup
\{\pm(e_7-e_8)\}
\]
together with
\[
\left\{
\frac12\left(\sum_{i=1}^6\varepsilon_i e_i+\varepsilon_7(e_7-e_8)\right)
\ \middle|\ 
\varepsilon_i=\pm 1,\ 
\prod_{i=1}^6\varepsilon_i=-1
\right\}.
\]
Thus \(|\Phi(E_7)|=126\).

We choose simple roots
\[
\begin{aligned}
\alpha_1&=\frac12(e_1-e_2-e_3-e_4-e_5-e_6-e_7+e_8),&
\alpha_2&=e_1+e_2,\\
\alpha_3&=-e_1+e_2,&
\alpha_4&=-e_2+e_3,\\
\alpha_5&=-e_3+e_4,&
\alpha_6&=-e_4+e_5,\\
\alpha_7&=-e_5+e_6.
\end{aligned}
\]
The Dynkin diagram is
\[
\begin{array}{ccccccccccc}
&&&& \alpha_2 &&&&&&\\
&&&& | &&&&&&\\
\alpha_1&-&\alpha_3&-&\alpha_4&-&\alpha_5&-&\alpha_6&-&\alpha_7 .
\end{array}
\]
The highest root is
\[
\theta
=
2\alpha_1+2\alpha_2+3\alpha_3+4\alpha_4+3\alpha_5+2\alpha_6+\alpha_7
=
e_8-e_7.
\]

The roots orthogonal to \(\theta\) are generated by
\[
\alpha_2,\alpha_3,\alpha_4,\alpha_5,\alpha_6,\alpha_7,
\]
which form a \(D_6\)-root system. Put
\[
\beta_1=\alpha_2,\qquad
\beta_2=\alpha_3,\qquad
\beta_3=\alpha_4,\qquad
\beta_4=\alpha_5,\qquad
\beta_5=\alpha_6,\qquad
\beta_6=\alpha_7.
\]
The intrinsic \(D_6\)-fundamental weights \(\varpi_i\in
\operatorname{span}_{\mathbb R}\{\beta_1,\dots,\beta_6\}\), defined by \((\varpi_i,\beta_j)=\delta_{ij},\) are
\[
\begin{aligned}
\varpi_1&=\frac12(e_1+e_2+e_3+e_4+e_5+e_6),\\
\varpi_2&=-\frac12e_1+\frac12(e_2+e_3+e_4+e_5+e_6),\\
\varpi_3&=e_3+e_4+e_5+e_6,\\
\varpi_4&=e_4+e_5+e_6,\\
\varpi_5&=e_5+e_6,\\
\varpi_6&=e_6.
\end{aligned}
\]
A dominant \(D_6\)-weight is
\[
\lambda=b_1\varpi_1+b_2\varpi_2+b_3\varpi_3
+b_4\varpi_4+b_5\varpi_5+b_6\varpi_6,
\qquad b_i\in\mathbb Z_{\geq 0}.
\]
If \(\theta^\vee(\mathbb G_m)\) acts with weight \(j\), then the corresponding
\(E_7\)-weight is
\[
\mu=\lambda+\frac j2\theta
=
\lambda+\frac j2(e_8-e_7).
\]
The numerical solutions are precisely those satisfying
\[
b_2=b_3=0.
\]
Thus
\[
\lambda=b_1\varpi_1+b_4\varpi_4+b_5\varpi_5+b_6\varpi_6,
\]
where \((b_1,b_4,b_5,b_6)\) appears in the following table:
\[
\begin{array}{c|c}
(b_1,b_4,b_5) & \text{allowed }b_6 \\ \hline
(0,0,0) & 0,\dots,7\\
(0,0,1) & 0,\dots,7\\
(0,0,2) & 0,\dots,7\\
(0,0,3) & 0,\dots,7\\
(0,1,0) & 0,\dots,8\\
(0,1,1) & 0,\dots,8\\
(0,1,2) & 0,\dots,8\\
(0,1,3) & 0,\dots,8
\end{array}
\qquad
\begin{array}{c|c}
(b_1,b_4,b_5) & \text{allowed }b_6 \\ \hline
(1,0,0) & 0,\dots,6\\
(1,0,1) & 0,\dots,6\\
(1,0,2) & 0,\dots,6\\
(1,0,3) & 0\\
(1,0,4) & 0\\
(2,0,0) & 0,\dots,5,\ 7\\
(3,0,0) & 0\\
\phantom{(3,0,0)} & \phantom{0}
\end{array}
\]
There are \(99\) such weights.

\ 

\textbf{$E_8$ case:}
We realize \(E_8\) in \(\mathbb R^8\) with orthonormal basis
\(e_1,\dots,e_8\). The roots are
\[
\Phi(E_8)
=
\{\pm e_i\pm e_j\mid 1\leq i<j\leq 8\}
\cup
\left\{
\frac12\sum_{i=1}^8\varepsilon_i e_i
\ \middle|\
\varepsilon_i=\pm1,\ 
\#\{i\mid \varepsilon_i=-1\}\equiv 0 \pmod 2
\right\}.
\]
We choose simple roots
\[
\begin{aligned}
\alpha_1&=\frac12(e_1-e_2-e_3-e_4-e_5-e_6-e_7+e_8),&
\alpha_2&=e_1+e_2,\\
\alpha_3&=-e_1+e_2,&
\alpha_4&=-e_2+e_3,\\
\alpha_5&=-e_3+e_4,&
\alpha_6&=-e_4+e_5,\\
\alpha_7&=-e_5+e_6,&
\alpha_8&=-e_6+e_7.
\end{aligned}
\]
The Dynkin diagram is
\[
\begin{array}{ccccccccccccc}
&&&& \alpha_2 &&&&&&&&\\
&&&& | &&&&&&&&\\
\alpha_1&-&\alpha_3&-&\alpha_4&-&\alpha_5&-&\alpha_6&-&\alpha_7&-&\alpha_8 .
\end{array}
\]
The highest root is
\[
\theta
=
2\alpha_1+3\alpha_2+4\alpha_3+6\alpha_4
+5\alpha_5+4\alpha_6+3\alpha_7+2\alpha_8
=
e_7+e_8.
\]
The roots orthogonal to \(\theta\) are generated by
\[
\alpha_1,\alpha_2,\alpha_3,\alpha_4,\alpha_5,\alpha_6,\alpha_7,
\]
which form an \(E_7\)-root system. Let
\[
\beta_i=\alpha_i,\qquad 1\leq i\leq 7.
\]
Let \(\varpi_1,\dots,\varpi_7\) be the intrinsic fundamental weights of this
\(E_7\)-subsystem, defined by
\[
(\varpi_i,\beta_j)=\delta_{ij},
\qquad
\varpi_i\in \operatorname{span}_{\mathbb R}\{\beta_1,\dots,\beta_7\}.
\]
Explicitly,
\[
\begin{aligned}
\varpi_1&=-e_7+e_8,\\
\varpi_2&=\frac12(e_1+e_2+e_3+e_4+e_5+e_6)-e_7+e_8,\\
\varpi_3&=-\frac12e_1+\frac12(e_2+e_3+e_4+e_5+e_6)
-\frac32e_7+\frac32e_8,\\
\varpi_4&=e_3+e_4+e_5+e_6-2e_7+2e_8,\\
\varpi_5&=e_4+e_5+e_6-\frac32e_7+\frac32e_8,\\
\varpi_6&=e_5+e_6-e_7+e_8,\\
\varpi_7&=e_6-\frac12e_7+\frac12e_8.
\end{aligned}
\]
A dominant \(E_7\)-weight is therefore
\[
\lambda
=
c_1\varpi_1+c_2\varpi_2+c_3\varpi_3+c_4\varpi_4
+c_5\varpi_5+c_6\varpi_6+c_7\varpi_7,
\qquad c_i\in\mathbb Z_{\geq 0}.
\]
The numerical solutions for maximal Cohen-Macaulay condition are precisely those satisfying
\[
c_4=c_5=c_6=c_7=0.
\]
Thus
\[
\lambda=c_1\varpi_1+c_2\varpi_2+c_3\varpi_3,
\]
where \((c_1,c_2,c_3)\) appears in the following triples: $(0,0,0)$, $(0,0,1)$, $(0,1,0)$, $(0,2,0)$,
$(1,0,0)$,$(1,0,1)$, $(1,1,0)$,
$(2,0,0)$, $(2,0,1)$, $(2,1,0)$,
$(3,0,0)$, $(3,0,1)$, $(3,1,0)$,
$(4,0,0)$, $(4,0,1)$, $(4,1,0)$,
$(5,0,0)$, $(5,0,1)$. There are \(18\) such weights.

\ 

\textbf{$F_4$ case:}
For \(F_4\), use the standard realization in \(\mathbb R^4\) with orthonormal basis
\(e_1,e_2,e_3,e_4\). Then
\[
\Phi(F_4)=
\{\pm e_i\mid 1\leq i\leq 4\}
\cup
\{\pm e_i\pm e_j\mid 1\leq i<j\leq 4\}
\cup
\left\{
\frac12(\pm e_1\pm e_2\pm e_3\pm e_4)
\right\}.
\]
Choose simple roots
\[
\alpha_1=e_2-e_3,\qquad
\alpha_2=e_3-e_4,\qquad
\alpha_3=e_4,\qquad
\alpha_4=\frac12(e_1-e_2-e_3-e_4).
\]
The Dynkin diagram is
\[
\alpha_1-\alpha_2\Rightarrow \alpha_3-\alpha_4,
\]
where \(\alpha_1,\alpha_2\) are long and \(\alpha_3,\alpha_4\) are short. The highest root is
\[
\theta
=
2\alpha_1+3\alpha_2+4\alpha_3+2\alpha_4
=
e_1+e_2.
\]
Since \(\theta\) is long, \(\theta^\vee=\theta\).
The roots orthogonal to \(\theta\) are generated by
\[
\alpha_2,\alpha_3,\alpha_4,
\]
which form a \(C_3\)-root system. Put
\[
\beta_1=\alpha_4,\qquad
\beta_2=\alpha_3,\qquad
\beta_3=\alpha_2.
\]
Then the intrinsic \(C_3\)-fundamental weights, defined by
\[
\langle \varpi_i,\beta_j^\vee\rangle=\delta_{ij},
\]
are
\[
\varpi_1=\frac12(e_1-e_2),\qquad
\varpi_2=\frac12(e_1-e_2+e_3+e_4),\qquad
\varpi_3=\frac12(e_1-e_2)+e_3.
\]
A dominant \(C_3\)-weight is
\[
\lambda=a_1\varpi_1+a_2\varpi_2+a_3\varpi_3,
\qquad a_i\in\mathbb Z_{\geq 0}.
\]
The numerical solutions for the maximal Cohen-Macaulay condition are precisely those with
\[
a_3=0,
\]
and
\[
(a_1,a_2)\in
\{(0,0),(0,1),(1,0),(1,1),(2,0),(2,1),
(3,0),(3,1),(4,0),(5,1)\}.
\]
There are \(10\) such weights.

\ 

\textbf{$G_2$ case:}
We realize \(G_2\) in
\[
\{x_1+x_2+x_3=0\}\subset \mathbb R^3
\]
by
\[
\alpha_1=e_1-e_2,\qquad
\alpha_2=-2e_1+e_2+e_3.
\]
Then
\[
(\alpha_1,\alpha_1)=2,\qquad
(\alpha_2,\alpha_2)=6,\qquad
(\alpha_1,\alpha_2)=-3.
\]
The positive roots are
\[
\alpha_1,\quad
\alpha_2,\quad
\alpha_1+\alpha_2,\quad
2\alpha_1+\alpha_2,\quad
3\alpha_1+\alpha_2,\quad
3\alpha_1+2\alpha_2.
\]
The highest root is
\[
\theta=3\alpha_1+2\alpha_2=-e_1-e_2+2e_3.
\]
The roots orthogonal to \(\theta\) are generated by \(\alpha_1\). Let \(\varpi\) be the intrinsic \(A_1\)-fundamental weight,we have
\[
\varpi=\frac12\alpha_1=\frac12(e_1-e_2).
\]
Thus a dominant \(A_1\)-weight is
\[
\lambda=a\varpi,\qquad a\in\mathbb Z_{\geq 0}.
\]
The maximal Cohen-Macaulay condition is computed as $a=0,1,2$. Equivalently, $\lambda=a\varpi,\ a\in\{0,1,2\}.$

\bibliographystyle{amsalpha}
\bibliography{references}
\end{document}